\documentclass[11pt,letterpaper]{article}
\usepackage[margin=1in]{geometry}
\usepackage{newtxtext}
\usepackage{amsmath}
\usepackage{amsthm}
\usepackage{newtxmath}
\usepackage{bm}
\usepackage{graphicx,xcolor}
\usepackage{microtype}
\usepackage{setspace}
\usepackage{endnotes}

\usepackage{enumitem}
\setlist{topsep=4pt,itemsep=3pt,parsep=0pt}
\usepackage{algorithm}
\usepackage{algpseudocode}
\usepackage{float}
\usepackage{booktabs}
\usepackage{tabularx}
\usepackage{array}
\usepackage[font=small,labelfont=bf,labelsep=period]{caption}
\usepackage{subfig}
\usepackage{placeins}
\usepackage{needspace}
\graphicspath{{figures/main/}{figures/appendix/}{tex/}}
\usepackage{natbib}
\bibpunct[, ]{(}{)}{,}{a}{}{,}

\newcommand{\BIBand}{and}
\usepackage[hidelinks,unicode]{hyperref}
\hypersetup{
  pdftitle={Learning Choice Model Trees for Feature-Based Multi-Product Pricing: Exact Optimization and Field Evidence},
  pdfauthor={Jiajie Zhang; Yanqiu Ruan; Xiao Jin; Chung Piaw Teo},
  pdfsubject={Feature-based pricing with choice model trees},
  pdfkeywords={market segmentation, feature-based pricing, choice model trees, multinomial logit, dynamic programming}
}
\usepackage{fancyhdr}
\newtheorem{theorem}{Theorem}
\newtheorem{lemma}{Lemma}
\newtheorem{proposition}{Proposition}

\BeforeBeginEnvironment{proposition}{\Needspace{9\baselineskip}}
\BeforeBeginEnvironment{theorem}{\Needspace{9\baselineskip}}
\DeclareMathOperator*{\argmax}{arg\,max}

\newcommand{\TABLE}[3]{%
  \centering\caption{#1}%
  {\small\setstretch{1.05}#2\par}%
  \smallskip{\footnotesize\setstretch{1.05}\raggedright #3\par}%
}
\makeatletter
\renewcommand\paragraph{\@startsection{paragraph}{4}{0pt}{6pt plus 2pt minus 1pt}{-0.6em}{\normalfont\normalsize\bfseries}}
\makeatother
\begin{document}
\raggedbottom
\thispagestyle{plain}
\begin{center}
{\fontsize{18}{22}\selectfont\bfseries
Learning Choice Model Trees for\\
Feature-Based Multi-Product Pricing:\\
Exact Optimization and Field Evidence\par}
\vspace{16pt}
\begin{tabular*}{\textwidth}{@{\extracolsep{\fill}}cccc@{}}
\large Jiajie Zhang\textsuperscript{1} &
\large Yanqiu Ruan\textsuperscript{1} &
\large Xiao Jin\textsuperscript{1} &
\large Chung Piaw Teo\textsuperscript{2}\\[6pt]
\small\href{mailto:jiajie_z@nus.edu.sg}{jiajie\_z@nus.edu.sg} &
\small\href{mailto:yanqiu.ruan@nus.edu.sg}{yanqiu.ruan@nus.edu.sg} &
\small\href{mailto:xiao.j@nus.edu.sg}{xiao.j@nus.edu.sg} &
\small\href{mailto:bizteocp@nus.edu.sg}{bizteocp@nus.edu.sg}
\end{tabular*}\par
\vspace{16pt}
{\small
\textsuperscript{1}Institute of Operations Research and Analytics, National University of Singapore\\
\textsuperscript{2}NUS Business School, National University of Singapore\par}
\end{center}
\vspace{12pt}
\begin{abstract}
Feature-based multi-product pricing uses customer characteristics to identify demand heterogeneity and tailor prices across products.
Choice model trees segment customers through interpretable feature rules and fit a demand model within each leaf.
Existing methods typically construct these trees greedily, selecting one myopic split at a time.
We develop optimal choice model trees with multinomial logit leaves (OCMT-MNL), jointly optimizing the tree and leaf models within a prescribed depth.
Our exact dynamic program derives closed-form Fenchel lower bounds during constrained Newton iterations and propagates them across nested and disjoint customer subsets, avoiding new fits and resuming unfinished fits without repeating completed work.
In synthetic experiments, it reduces exact leaf fits by 99.98\% and leaf evaluations by 86.13\%, achieving up to 7.15-fold speedups over unpruned dynamic programming.
One-dimensional lookup tables translate offline estimation into real-time pricing, with a revenue-loss bound quadratic in grid spacing under the fitted model.
Compared with greedy trees, OCMT-MNL achieves lower revenue loss with fewer leaves on synthetic data and better predictive fit on real data.
In a 23-week randomized experiment on ancillary seat pricing across 48 airline markets and 190{,}220 passengers, OCMT-MNL increases seat revenue per passenger by a statistically significant 11.3\% over static pricing.
\end{abstract}


\section{Introduction}\label{sec:intro}

Firms often observe customer and transaction information before quoting prices. Feature-based pricing uses this information to tailor prices to heterogeneous demand \citep{BergemannBrooksMorris2015,ElmachtoubGuptaHamilton2021}. The problem becomes substantially harder when a firm prices a menu rather than a single product. Customer characteristics may affect preferences and price sensitivity, while changing one price can shift demand elsewhere in the menu. The firm must therefore learn how observed features shape demand across products and translate that knowledge into a price vector at the time of offer. Airline ancillary pricing is one prominent example. Before displaying an offer, an airline observes trip and booking information and must price a menu of ancillary products in real time. IdeaWorksCompany estimated worldwide airline ancillary revenue at \$157 billion for 2025, representing 15.7\% of total airline revenue \citep{IdeaWorks2025GlobalEstimate}. The same challenge arises more broadly wherever firms must capture heterogeneous multi-product demand in an intelligible form and translate it into prices at decision time.

Discrete choice models provide a natural foundation for multi-product demand. Under the random utility framework, a customer chooses the alternative with the highest utility, allowing a price change to shift demand across the product menu. Among these models, the multinomial logit (MNL) model is widely used because it is parsimonious and tractable for both estimation and pricing \citep{TalluriVanRyzin2004,Wang2021}. A single MNL, however, imposes one set of demand parameters on the entire population and may mask heterogeneity in preferences and price sensitivity. Latent class MNL models address this limitation by allowing demand parameters to vary across unobserved customer classes \citep{KamakuraRussell1989,GreeneHensher2003}, and observed characteristics can be used to predict class membership probabilities \citep{KamakuraWedelAgrawal1994}. Class membership nevertheless remains latent and probabilistic, so the resulting segments do not directly translate into deterministic and auditable rules based on observable customer characteristics.

Interpretability has become an increasingly important consideration in feature-based pricing. As firms use richer customer information to differentiate offers, researchers and regulators have paid closer attention to how pricing rules can be understood and examined \citep{CohenElmachtoubLei2022,DubeMisra2023}.\endnote{For recent regulatory attention to the use of consumer data in customer segmentation and targeted pricing, see \citet{FTC2025}.} A transparent representation allows managers to see which observable characteristics distinguish customers receiving different prices and whether those distinctions correspond to plausible differences in demand \citep{Rudin2019,Biggsetal2021}. It can also make unexpected recommendations easier to investigate and clarify how customer information enters pricing decisions. These benefits make a compact and explicit representation of heterogeneous demand particularly attractive for feature-based pricing.

Choice model trees offer a natural way to obtain such a representation. Unlike standard classification or regression trees, whose leaves contain class labels or constant predictions, a choice model tree places a structured choice model in each leaf. Its internal nodes partition customers using observable feature rules, while each leaf describes the demand behavior of the resulting segment \citep{Misic2016,Aouadetal2023}. A root-to-leaf path therefore provides an explicit rule for segment membership, and the corresponding leaf model captures how preferences and price sensitivity vary across the product menu. This structure allows managers to examine both how customers are segmented and how each segment responds to prices.

Nevertheless, the interpretability of a choice model tree depends critically on its depth. As a tree grows deeper, the number of segments and the length of its routing rules increase rapidly, making the resulting segmentation harder to understand and audit \citep{BertsimasDunn2017,Biggsetal2021}. Existing methods construct choice model trees through greedy recursive partitioning, selecting one split at a time \citep{Misic2016,Aouadetal2023}. Because early splits cannot be reconsidered, a locally attractive choice may force later levels to compensate, requiring more leaves to represent heterogeneity that a different tree could capture more compactly. The central question is therefore how to use a limited depth as effectively as possible to uncover an accurate customer segmentation. This calls for optimizing the tree structure and its leaf choice models jointly rather than split by split.

We address this problem by developing optimal choice model trees with MNL leaves (OCMT-MNL). For a prescribed depth, OCMT-MNL jointly selects the feature-based splits and estimates the MNL demand model in each leaf. We develop an exact algorithm that makes this global estimation computationally tractable. We also bridge offline estimation and real-time deployment. Once the tree has been fitted from historical data, an arriving customer is routed to a leaf, and the corresponding multi-product price vector is obtained from a leaf-specific one-dimensional lookup table indexed by a scalar customer score. Our synthetic experiments show that OCMT-MNL can be estimated substantially faster than an unpruned exact method and yields more compact, better-performing trees than greedy construction. We further deploy the resulting pricing policy in a 23-week randomized experiment with an Asia-Pacific airline, where it increases seat revenue per passenger by 11.3\% relative to the airline's static pricing baseline.

\subsection{Contributions}\label{sec:contributions}

Our contributions are threefold.

\begin{itemize}
\item \textbf{Exact OCMT-MNL estimation at practical cost.}
We formulate OCMT-MNL estimation as a depth-limited global tree problem and develop a memoized dynamic program that returns its global optimum (Theorem~\ref{thm:dp}). To handle the repeated constrained MNL estimation required across candidate leaves, we make leaf optimization part of the tree search itself. At intermediate constrained Newton iterates, quantities already computed for fitting yield closed-form Fenchel lower bounds before convergence (Proposition~\ref{prop:newton_linear_bound}), and we characterize their convergence and local tightness (Proposition~\ref{prop:local_accuracy}). If a bound does not prune, fitting continues from the same solver state. Cross-set propagation further transfers the resulting bound information across nested and disjoint datasets, allowing work on one leaf to prevent another leaf fit from starting (Proposition~\ref{prop:cross_set_propagation}). A four-stage ablation isolates these mechanisms. The full method reduces exact leaf fits by 99.98\% and MNL context evaluations by 86.13\%, yielding up to a 7.15$\times$ computational speedup over the unpruned dynamic program.

\item \textbf{From offline estimation to real-time pricing.}
We characterize the pricing problem induced by each fitted leaf and reduce it to a scalar fixed point, yielding a unique price map that is monotone and Lipschitz in a scalar feature index (Proposition~\ref{prop:pricing_structure}). Within a correctly specified leaf, we further show that likelihood fit controls pricing regret, providing a decision-theoretic rationale for the estimation objective (Theorem~\ref{thm:likelihood_pricing}). These results turn the tree estimated offline into a real-time pricing rule requiring only tree routing, evaluation of the scalar index, and a query to the leaf's one-dimensional lookup table. Relative to exact within-leaf optimization, the resulting revenue loss decreases quadratically as the lookup grid is refined (Proposition~\ref{prop:lookup}).

\item \textbf{Model comparison and field validation.}
Across three synthetic regimes that vary the agreement between locally and globally preferred splits, OCMT-MNL outperforms the greedy choice model tree with MNL leaves (GCMT-MNL) \citep{Aouadetal2023} in predictive fit, structural recovery, and downstream pricing. On historical data from 48 airline markets, OCMT-MNL attains a statistically significant out-of-sample likelihood improvement over GCMT-MNL and performs better in 35 markets. These comparisons evaluate global versus greedy tree construction. Separately, a 23-week randomized experiment covering 190{,}220 passengers evaluates the complete OCMT-MNL pricing policy against the airline's static baseline. OCMT-MNL increases seat revenue per passenger by 11.3\%, with positive market-level differences in 35 markets. The gain is associated mainly with a higher share of passengers in positive-revenue bookings and more premium seats sold per passenger.

\end{itemize}

\subsection{Related Literature}\label{sec:literature}

Our work lies at the intersection of feature-based pricing and optimal trees.

\paragraph{Feature-Based Pricing and Market Segmentation.}
One stream of feature-based pricing treats pricing as a sequential learning problem: the seller observes customer features, posts a price, and learns from the purchase decision. This literature studies binary feedback, high-dimensional features, heterogeneous price sensitivity, and model misspecification \citep{CohenLobelPaesLeme2020,JavanmardNazerzadeh2019,BanKeskin2021,NambiarSimchiLeviWang2019}; \citet{denBoer2015} provides a survey. Multi-product extensions combine online learning with choice models, including MNL-bandit assortment optimization and contextual assortment and pricing \citep{AgrawalMNLBandit2019,Javanmardetal2020,GoyalPerivier2022}. Their objective is to balance learning and earning over time. OCMT-MNL instead uses an offline sample to recover a compact segmentation and a choice model for each segment.

A complementary stream estimates demand from historical offers and choices and then optimizes decisions against the fitted model. \citet{ChenOwenPixtonSimchiLevi2022}, for example, study feature-based pricing and assortment under logit demand and establish finite-sample revenue guarantees. A related line of work asks how estimation error affects downstream decisions. \citet{XuWang2021} bound single-product pricing regret by excess expected negative log-likelihood; \citet{TanFrazier2022} derive local quadratic regret bounds for smooth estimate-then-optimize problems; and \citet{Zhangetal2026ContextualMNL} construct pessimistic offline policies for contextual MNL assortment and pricing when the logged data provide suitable local coverage. Our likelihood-to-revenue result contributes to this line by treating plug-in multi-product MNL pricing under linear parameter constraints and price bounds.

These studies use customer features to personalize decisions, but compact market segmentation is not their primary object. \citet{CuiHamilton2026} make segmentation central by mapping features to a scalar valuation and jointly selecting segments and prices for a single product; under independent monotone-hazard-rate regression errors, dynamic programming recovers interval segments in that valuation. OCMT-MNL instead estimates segmentation and segment-level demand jointly from multi-product choice data. Prescriptive trees map features directly to treatments or decisions \citep{Kallus2017,Bertsimasetal2019,ElmachtoubLiangMcNellis2020}; our leaves retain a choice model for evaluating continuous multi-product price vectors. Neither approach globally constructs a feature-rule tree with fitted multi-product choice models.

\paragraph{Optimal Trees.}
Greedy classification and regression tree (CART) algorithms choose one split at a time and do not optimize the resulting tree as a whole \citep{Breimanetal1984}. Exact classification tree methods follow two main approaches. Mixed-integer optimization formulations construct the complete tree jointly \citep{BertsimasDunn2017,Aghaeietal2025}. Specialized dynamic programming and branch-and-bound algorithms instead solve repeated subtree problems once, cache their values, and use lower bounds to avoid unpromising branches \citep{Aglinetal2020,Demirovicetal2022}. \citet{vanderLindenetal2023} characterize when a tree objective and its constraints permit optimal subtrees to be solved independently and combined by dynamic programming, and \citet{Britaetal2025} extend exact dynamic programming to continuous split thresholds. These algorithms benefit from the structure of classification leaves: the best label and its loss follow directly from class counts, making both leaf evaluation and many lower bounds inexpensive.

That advantage weakens when a leaf contains a fitted model. Early model tree methods, including M5 and model-based recursive partitioning, fit linear or more general parametric models in the leaves but construct the tree through local recursive splitting \citep{Quinlan1992,Zeileisetal2008}. \citet{Dunn2018} develops mixed-integer formulations for globally optimized regression trees with constant or linear leaf predictions, but also documents their limited scalability. The closest work on dynamic programming is \citet{vandenBosetal2024}, who develop algorithms that globally optimize the structure of piecewise constant, simple linear, and multiple linear regression trees. For constant and simple linear leaves, the fitted loss can be recovered from additive sufficient statistics. This property supports a specialized depth-two routine that avoids fitting every candidate leaf separately. For multiple linear elastic-net leaves, the paper fits each leaf by coordinate descent and leaves an efficient per-instance cost breakdown open. Consequently, no analogous depth-two routine is available. This distinction is important for our setting because an MNL leaf is also a multivariate estimation problem rather than a count-based terminal cost.

Choice model trees face a similar computational challenge because evaluating a candidate leaf requires fitting its choice model. The methods closest to OCMT-MNL are constructed greedily. \citet{Misic2016} gives an early procedure that splits on customer features and fits an MNL model in each terminal segment. \citet{Aouadetal2023} introduce the broader Market Segmentation Tree framework, in which splits are selected to improve the fit of response models rather than to group customers solely by feature similarity. Their Choice Model Tree specialization fits an MNL model in every leaf and uses recursive partitioning to choose the locally best split; its implementation uses parallel computation and warm starts for scalability. The method jointly learns a segmentation and its response models, but it does not optimize the tree globally or address downstream pricing.

\citet{ElmachtoubGoutamLederman2025} connect this modeling approach to pricing scheduled services at Amazon. Their system uses a Market Segmentation Tree and modified MNL leaf models with local reference-price effects. Its pricing heuristics search in a small number of dimensions and are designed for short response times. It therefore provides the closest end-to-end example of tree-based choice modeling and multi-product pricing, while retaining the greedy tree-construction step inherited from the Market Segmentation Tree framework. These two lines leave a specific gap. Exact tree methods do not address fitted MNL demand models, while choice model trees retain greedy construction. To our knowledge, OCMT-MNL is the first exact method that jointly optimizes a customer segmentation tree of prescribed depth and the MNL demand models fitted at its leaves.

\subsection{Organization}\label{sec:organization}
The remainder of the paper is organized as follows. Section~\ref{sec:formulation} formulates the choice model tree and the segment-level pricing problem. Section~\ref{sec:dp} defines the estimation problem and develops the exact dynamic program and its pruning rules. Section~\ref{sec:solution} characterizes within-leaf pricing, relates likelihood fit to pricing loss, and develops the rule for online implementation. Section~\ref{sec:synthetic} evaluates the computational gains from pruning and compares the optimal and greedy choice model trees with respect to both predictive and prescriptive performance. Section~\ref{sec:field} presents the large-scale field experiment results for feature-based seat pricing with a major Asia-Pacific airline. Finally, Section~\ref{sec:conclusions} concludes the paper. The appendices provide the complete algorithms and proofs, the synthetic data-generating process, an out-of-sample comparison on historical airline data, and additional details on field-experiment sampling and the learned segmentations.

\section{Model and Problem Formulation}\label{sec:formulation}

OCMT-MNL combines a feature-based partition, an MNL demand model within each segment, and the pricing problem induced by that model. We formalize these objects under common depth and leaf-size constraints. The formulation separates the offline task of estimating the tree from the segment-level pricing problem solved after a customer has been routed to a leaf.

\subsection{Choice Model Trees with MNL Leaves}\label{sec:mob}
A firm sells a set of products $\mathcal{J}$ to a heterogeneous population of customers. The firm observes historical transaction data $\mathcal{D} = \{(\bm{x}_i, \bm{p}_i,y_i)\}_{i\in\mathcal{I}}$, with index set $\mathcal{I}$ and sample size $N = |\mathcal{I}|$. Here $\bm{x}_i \in \mathcal{X}$ is the customer feature vector, $\bm{p}_i \in \mathbb{R}^{|\mathcal{J}|}_{+}$ the prices offered to customer~$i$, and $y_i \in \{0\}\cup\mathcal{J}$ the observed choice, with $y_i=0$ denoting no purchase. We take the operational feature domain~$\mathcal X$ to be compact.

Leaf utilities use covariates $\mathcal F=\mathcal F_{\textup{cont}}\cup\mathcal F_{\textup{bin}}$. The set $\mathcal F_{\textup{cont}}$ contains numerical covariates, and $\mathcal F_{\textup{bin}}$ contains binary covariates, including indicators for categories of categorical variables. Tree branches use a finite set of split candidates $\mathcal F'$ derived from them: binary indicators enter directly and numerical covariates through threshold indicators. This separation keeps the branch search finite and interpretable without increasing the dimension of the leaf model. Thus $\mathcal X\subset\mathbb R^{|\mathcal F|}$.

The full-tree choice model is a map $h:\mathcal{X}\times \mathbb{R}_{+}^{|\mathcal{J}|} \to \Delta_{\mathcal{J}}$, where $\Delta_{\mathcal{J}}:=\{\bm{\rho}\in\mathbb{R}_{+}^{|\mathcal{J}|+1}:\sum_{j\in\{0\}\cup\mathcal{J}}\rho_j=1\}$ is the probability simplex over the no-purchase option and the products. Its component $h_j(\bm{x},\bm{p})$ is the predicted probability of option~$j$. OCMT-MNL represents $h$ by a binary tree whose leaves carry independent MNL parameters. Let $\mathfrak T(\mathcal F',D,N_{\min})$ denote the class of trees with maximum depth~$D$, splits in~$\mathcal F'$, and at least~$N_{\min}$ observations per leaf.

\begin{figure}[htb]
\centering
\includegraphics[width=0.82\linewidth]{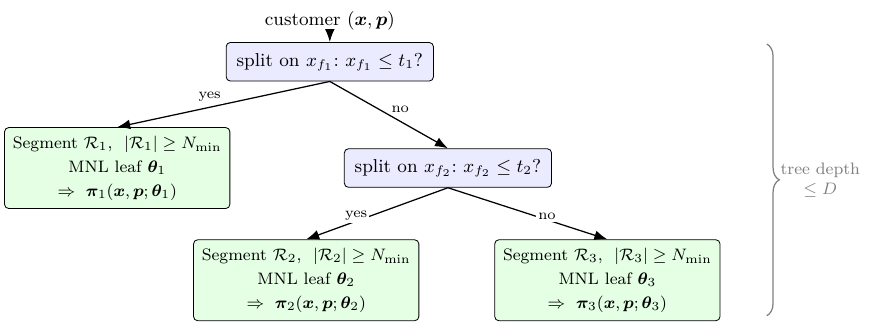}
\caption{The OCMT-MNL choice model tree: internal nodes split on candidates in~$\mathcal{F}'$ up to depth~$D$; each leaf is a market segment~$\mathcal{R}_k$ ($\ge N_{\min}$ observations) carrying its own MNL with parameters $\bm{\theta}_k = (\bm{\alpha}_k, \bm{\beta}_k, \bm{\gamma}_k)$.\label{fig:model_tree}}
\end{figure}

A feasible tree $\mathbb{T} \in \mathfrak{T}(\mathcal{F}', D, N_{\min})$ partitions the feature space $\mathcal{X}$ into $K(\mathbb{T})$ disjoint regions $\{\mathcal{R}_k(\mathbb{T})\}_{k=1}^{K(\mathbb{T})}$. Each region is a market segment, and the number of segments $K(\mathbb{T})$ is determined by the tree rather than fixed a priori. Within segment~$k$, utility is $u_j(\bm{x},\bm{p},\bm{\theta}_k)=\alpha_{jk}+\bm{\beta}_k^\top\bm{x}-\gamma_{jk}p_j$, where $\bm{\theta}_k = (\bm{\alpha}_k, \bm{\beta}_k, \bm{\gamma}_k)$, $\bm{\alpha}_{k} \in \mathbb{R}^{|\mathcal{J}|}$ is the vector of product-specific intercepts, $\bm{\beta}_{k} \in \mathbb{R}^{|\mathcal{F}|}$ is the common feature coefficient in the leaf, and $\bm{\gamma}_{k} \in [\underline\gamma,\infty)^{|\mathcal{J}|}$, with $\underline\gamma>0$, is the vector of product-specific price sensitivities. This shared feature coefficient is the structural restriction used in Section~\ref{sec:solution} to reduce within-leaf pricing to a scalar-index problem. Let $N_0:=2|\mathcal J|+|\mathcal F|$ be the number of parameters in each leaf.

The leaf MNL probability vector is $\bm\pi_k(\bm x,\bm p;\bm\theta_k)=(\pi_{jk}(\bm x,\bm p;\bm\theta_k))_{j\in\{0\}\cup\mathcal J}$, with components
\begin{equation}
\pi_{jk}(\bm{x},\bm{p};\bm{\theta}_k) =
\begin{cases}
\displaystyle \frac{1}{1+\sum_{\ell\in\mathcal{J}}\exp\{u_\ell\}},  &  j=0, \\[8pt]
\displaystyle \frac{\exp\{u_j\}}{1+\sum_{\ell\in\mathcal{J}}\exp\{u_\ell\}},  &  j\in\mathcal{J}.
\end{cases}
\end{equation}
All leaves share the nonempty closed polyhedron $\Theta = \{\bm{\theta} \in \mathbb{R}^{N_0} : \bm A_{\Theta}\bm{\theta} \le \bm b_{\Theta}\}$. Linear constraints are important in pricing applications because economically meaningful demand requires positive price sensitivities. We therefore include $\bm\gamma\ge\underline\gamma\bm 1$ in~$\Theta$; the same representation accommodates other application-specific bounds and shape restrictions. We assume that every finite leaf problem attains its minimum; compactness of~$\Theta$ is sufficient. Both the number of segments $K(\mathbb{T})$ and the parameter collection $\{\bm\theta_k\}_{k=1}^{K(\mathbb{T})}$ are induced by the tree structure
$\mathbb{T}$; when this dependence is clear from context, we suppress it to
lighten notation.

Given a choice model tree $(\mathbb{T}, \{\bm{\theta}_k\}_{k=1}^{K(\mathbb{T})})$, the induced choice model is
\begin{equation}
	h(\bm{x},\bm{p})
	=
	\sum_{k=1}^{K(\mathbb{T})}
	\mathbb{I}\{\bm{x}\in \mathcal{R}_k(\mathbb{T})\}
	\bm\pi_k(\bm{x},\bm p;\bm\theta_k).
\end{equation}
This choice model tree with MNL leaves is the \emph{OCMT-MNL model}; Section~\ref{sec:estimation} develops its estimation from data.

\subsection{Segment-Level Pricing}\label{sec:pricing_problem}
Given a fitted choice model tree $(\mathbb{T}, \{\bm{\theta}_k\}_{k=1}^{K(\mathbb{T})})$, fix a customer feature vector~$\bm{x}$ and let $k=k(\bm{x};\mathbb{T})$ denote its unique leaf. Conditional on this leaf, the prescribed price vector~$\bm p_k^*(\bm x)$ solves
\begin{equation}\label{obj:revenue}
\bm p_k^*(\bm{x})\in
\arg\max_{\bm{p}}
\left\{
\sum_{j\in\mathcal{J}}
p_j\,
\frac{\exp\{u_j(\bm{x},\bm{p},\bm{\theta}_k)\}}
{1+\sum_{\ell\in\mathcal{J}}\exp\{u_\ell(\bm{x},\bm{p},\bm{\theta}_k)\}}
\;:\;
\underline{p}_j \le p_j \le \bar{p}_j,\ \forall j\in\mathcal{J}
\right\}.
\end{equation}
Here $0\le\underline{p}_j<\bar{p}_j<\infty$ and $\mathcal P:=\prod_{j\in\mathcal J}[\underline p_j,\bar p_j]$. Distinct from the parameter constraints defining~$\Theta$, these price bounds encode business, regulatory, or experimental restrictions. Section~\ref{sec:solution} characterizes~\eqref{obj:revenue} and develops an efficient lookup-table scheme.

\section{Dynamic Programming for Optimal Choice Model Trees}\label{sec:dp}

Exact OCMT-MNL estimation couples a discrete search over trees with repeated constrained MNL fits. A direct Bellman recursion is exact, but it treats each leaf cost as available only after numerical convergence. We instead make the intermediate optimization path available to the recursion. Partial leaf fits yield lower bounds, and those bounds propagate across related datasets and Bellman states. The resulting search can eliminate a leaf before fitting begins or stop a fit early, while preserving the work needed if the leaf remains competitive.

\subsection{Penalized Estimation Model}\label{sec:estimation}
We estimate the tree and its leaf parameters by minimizing penalized negative log-likelihood, with penalty $\lambda\ge0$ and minimum leaf size $N_{\min}\in\{1,\ldots,N\}$. For a candidate tree $\mathbb{T}\in\mathfrak{T}(\mathcal{F}',D,N_{\min})$, let $\mathcal{I}_k(\mathbb{T})=\{i\in\mathcal{I}:\bm{x}_i\in\mathcal{R}_k(\mathbb{T})\}$ denote the observations assigned to leaf~$k$, and let $\mathcal D_k(\mathbb{T})=\{(\bm x_i,\bm p_i,y_i):i\in\mathcal I_k(\mathbb{T})\}$ be the corresponding leaf dataset. Dropping the leaf index, write $\bm\pi_i(\bm\theta)=(\pi_{ij}(\bm\theta))_{j\in\{0\}\cup\mathcal J}$ and $\ell_i(\bm\theta)=-\log\pi_{i,y_i}(\bm\theta)$ for the MNL probability vector and negative log-likelihood of observation~$i$ under $\bm\theta\in\Theta$.

For any subset $\mathcal{D}_s\subseteq\mathcal{D}$ with index set $\mathcal{S}$, define its penalized leaf cost as
\begin{equation}\label{eq:leaf_cost}
H(\mathcal{D}_s)
=
\begin{cases}
\displaystyle
\min_{\bm{\theta}\in\Theta}
\left\{
\sum_{i\in\mathcal{S}}\ell_i(\bm{\theta})
+
\lambda N_0
\right\}
& |\mathcal{D}_s|\ge N_{\min},\\[6pt]
+\infty,
& \text{otherwise}.
\end{cases}
\end{equation}
The MNL negative log-likelihood is convex in $\bm{\theta}$, and $\Theta$ is the common polyhedral feasible set defined in Section~\ref{sec:mob}. Each finite leaf cost is therefore a constrained convex estimation problem. The OCMT-MNL estimation problem is
\begin{equation}\label{eq:ocmt_mnl_est}
\textup{[OCMT-MNL-EST]}\qquad
\min_{\substack{\mathbb{T}\in\mathfrak{T}(\mathcal{F}',D,N_{\min})\\
\bm{\theta}_k\in\Theta,\ k=1,\ldots,K(\mathbb{T})}}
\left\{
\sum_{k=1}^{K(\mathbb{T})}
\sum_{i\in\mathcal{I}_k(\mathbb{T})}
\ell_i(\bm{\theta}_k)
+
\lambda K(\mathbb{T})N_0
\right\}.
\end{equation}

The objective in~\eqref{eq:ocmt_mnl_est} is predictive; Section~\ref{sec:pricing} connects it to downstream pricing.

With the negative log-likelihood objective in~\eqref{eq:ocmt_mnl_est}, the AIC- and BIC-based penalties used below are $\lambda_{\textup{AIC}}=1$ and $\lambda_{\textup{BIC}}(n)=\tfrac12\log n$, respectively, for an estimation sample of size~$n$.

For a fixed tree, leaf parameters are uncoupled, so its value is the sum of the leaf costs $H(\mathcal{D}_k(\mathbb{T}))$.

\paragraph{Leaf solver and warm starts.}
We evaluate finite leaf costs by constrained Newton, implemented as sequential quadratic programming (SQP) with the exact Hessian and a feasible initial point \citep{NocedalWright2006}. At a feasible iterate~$\bm\theta_t$, let $\bm g_t$ and $\bm G_t$ be the gradient and Hessian of the unpenalized leaf negative log-likelihood. The SQP direction~$\bm d_t$ solves the convex quadratic model
\begin{equation}\label{eq:constrained_newton_qp}
\min_{\bm d}
\left\{
\bm g_t^\top\bm d+\frac12\bm d^\top\bm G_t\bm d
\;:\;
\bm A_\Theta(\bm\theta_t+\bm d)\le\bm b_\Theta
\right\},
\end{equation}
and returns its Karush--Kuhn--Tucker (KKT) multipliers~$\bm\mu_t$ from the same QP solve. A damped line search along~$\bm d_t$ ensures sufficient decrease. Because both endpoints are feasible and~$\Theta$ is convex, every trial point remains feasible. All leaves share~$\Theta$, so a fitted parent parameter is a feasible initial point for either child and is used whenever available; the root starts from a fixed feasible point. This warm start changes only the initial iterate, not the leaf optimum. Each SQP iteration evaluates the MNL probabilities, gradient, and Hessian, which we call an \emph{MNL context evaluation}. Section~\ref{sec:acceleration} uses the same context, direction, and multipliers to construct a Fenchel bound; if the bound does not prune, constrained Newton continues from the stored solver state.

\subsection{Exact Bellman Recursion}\label{sec:bellman}

For any customer subset, the search compares stopping at one MNL leaf with splitting and selecting the best admissible subtree on each side. Each split is therefore assessed together with its continuation within the remaining depth, rather than by its immediate improvement in fit. Because this decision depends only on the subset and remaining depth, it admits a Bellman recursion.

We instantiate the finite set of split candidates introduced in Section~\ref{sec:mob} as follows. All covariates in $\mathcal F_{\textup{bin}}$ are retained as split candidates. After scaling numerical features to $[0,1]$, choose a finite threshold set $\mathcal T_f\subset[0,1]$ for each $f\in\mathcal F_{\textup{cont}}$, and let each $t\in\mathcal T_f$ induce the binary split indicator $\mathbb I\{x_f\le t\}$. Thus $\mathcal{F}'=\mathcal{F}_{\textup{bin}}\cup\{(f,t):f\in\mathcal F_{\textup{cont}},\ t\in\mathcal T_f\}$.
For each split candidate $f\in\mathcal F'$, let $z_{if}\in\{0,1\}$ denote the resulting binary value for observation~$i$. With this finite split set fixed, the dynamic program optimizes over the depth-limited hypothesis class $\mathfrak{T}(\mathcal F',D,N_{\min})$.

For any subset $\mathcal D_s=\{(\bm{x}_i,\bm p_i,y_i)\}_{i\in\mathcal S}$ and depth $d\in\{0,\ldots,D\}$, define the state $(\mathcal D_s,d)$. Let $V(\mathcal D_s,d)$ be the minimum penalized negative log-likelihood attainable by any feasible subtree rooted at this state with remaining depth $D-d$. Splitting on $f\in\mathcal F'$ creates $\mathcal D_s^{f=b}:=\{(\bm x_i,\bm p_i,y_i)\in\mathcal D_s:z_{if}=b\}$ for $b\in\{0,1\}$. The admissible set is $\mathcal A(\mathcal D_s,d):=\{f\in\mathcal F':d<D,\ \min\{|\mathcal D_s^{f=0}|,|\mathcal D_s^{f=1}|\}\ge N_{\min}\}$.
With the convention that the minimum over an empty set is $+\infty$, the unpruned Bellman recursion is
\begin{equation}\label{eq:dp}
V(\mathcal D_s,d)
=
\min\Bigl\{
H(\mathcal D_s),\;
\min_{f\in\mathcal A(\mathcal D_s,d)}
\left\{
V(\mathcal D_s^{f=0},d+1)+V(\mathcal D_s^{f=1},d+1)
\right\}
\Bigr\},
\end{equation}
and $V(\mathcal D_s,d)=+\infty$ whenever $|\mathcal D_s|<N_{\min}$. In particular, at depth $D$ the admissible split set is empty and the recursion reduces to $V(\mathcal D_s,D)=H(\mathcal D_s)$. The optimal value of \textup{[OCMT-MNL-EST]} is $V(\mathcal D,0)$.

\begin{theorem}[Optimality and complexity]\label{thm:dp}
If each finite leaf cost is evaluated exactly, Algorithm~\ref{alg:dp_compact_final} returns an optimal solution of \textup{[OCMT-MNL-EST]} over $\mathfrak T(\mathcal F',D,N_{\min})$. When the leaf solver is run to numerical tolerance~$\varepsilon$, let $m_{\Theta}$ be the number of linear constraints, $\overline m_{\textup{Newt}}(N,\varepsilon)$ the maximum number of constrained Newton iterations for a feasible leaf with at most $N$ observations, and $\mathcal T_{\textup{QP}}(N_0,m_\Theta,\varepsilon)$ the cost of one constrained Newton QP solve. The running time is
\[
O\!\left(
\left\{\sum_{d=0}^{D}\binom{|\mathcal F'|}{d}2^d\right\}
\left[
\overline m_{\textup{Newt}}(N,\varepsilon)
\left\{NN_0^2+m_{\Theta}N_0+
\mathcal T_{\textup{QP}}(N_0,m_\Theta,\varepsilon)\right\}
+|\mathcal F'|N
\right]
\right).
\]
\end{theorem}

The first factor bounds the number of candidate subsets, while the second accounts for constrained MNL fitting and split evaluation on each subset. Memoization solves each state $(\mathcal D_s,d)$ once, and backtracking the cached actions recovers the tree. Appendix~\ref{app:dp} gives the proof. Algorithm~\ref{alg:dp_compact_final} gives the unpruned recursion; Appendix~\ref{app:algorithms} gives its incumbent-capped counterpart.

\begin{algorithm}[!t]
\caption{Unpruned memoized DP for solving \textup{[OCMT-MNL-EST]}}
\label{alg:dp_compact_final}
\vskip2pt
{\setstretch{1}\footnotesize
\begin{algorithmic}[1]
\Require Dataset $\mathcal{D}$, candidate split set $\mathcal{F}'$, max depth $D$, minimum leaf size $N_{\min}$, penalty $\lambda$, and parameter set $\Theta$
\Statex
\Ensure Optimal tree $\mathbb{T}^*$ with leaf parameters $\{\bm{\theta}^*_k\}$
\Statex
\Function{SolveState}{$\mathcal{D}_s,d$}
\If{$(\mathcal{D}_s,d)$ is cached} \Return cached value
\EndIf
\If{$|\mathcal{D}_s|<N_{\min}$} cache $+\infty$; \Return $+\infty$
\EndIf
\State Compute or retrieve leaf cost $H(\mathcal{D}_s)$ and parameter $\bm{\theta}^*(\mathcal{D}_s)$
\State $v^*\leftarrow H(\mathcal{D}_s)$;\quad $a^*\leftarrow\textsf{Leaf}$
\If{$d<D$}
\ForAll{$f\in\mathcal{A}(\mathcal{D}_s,d)$}
\State $v_f\leftarrow\Call{SolveState}{\mathcal{D}_s^{f=0},d+1}+\Call{SolveState}{\mathcal{D}_s^{f=1},d+1}$
\If{$v_f<v^*$}
\State $v^*\leftarrow v_f$;\quad $a^*\leftarrow f$
\EndIf
\EndFor
\EndIf
\State Cache $(v^*,a^*)$ for $(\mathcal{D}_s,d)$ and the leaf fit for $\mathcal S$
\State \Return $v^*$
\EndFunction
\Statex
\State \Call{SolveState}{$\mathcal{D},0$}
\State Recover $(\mathbb{T}^*,\{\bm{\theta}^*_k\})$ by backtracking cached actions from root $(\mathcal{D},0)$
\end{algorithmic}
}
\end{algorithm}

\subsection{Branch-and-Bound Pruning}\label{sec:acceleration}

Lower bounds are central to exact tree search. Mixed-integer formulations use relaxations for optimal classification trees \citep{BertsimasDunn2017,Aghaeietal2025}, while DL8.5 and MurTree combine branch-and-bound with dynamic programming caches \citep{Aglinetal2020,Demirovicetal2022}; recent dynamic programs also accommodate continuous thresholds \citep{Britaetal2025}. These methods benefit from terminal losses obtained directly from class counts. In OCMT-MNL, by contrast, every candidate leaf requires constrained MNL estimation. For our search, bounds should become available before a leaf fit is complete. If a bound does not prune, fitting should continue from the same solver state without repeating the numerical work at that checkpoint. Bounds should also transfer across related leaves, allowing a candidate leaf to be pruned before its first MNL evaluation.

\pagebreak[3]
Related techniques use dual solutions of continuous subproblems in outer approximation \citep{CoeyLubinVielma2020} and iterate-based bounds or primal--dual gaps in statistical optimization \citep{Kohetal2007,Ndiayeetal2017}. Our Bellman recursion generates related MNL problems on different observation sets. We extract a bound from an intermediate constrained Newton iterate, continue the same path when pruning fails, and propagate its dual information to leaves whose fitting has not begun.

\subsubsection{Fenchel Leaf Bounds}

We begin with a single leaf. Let $\bm e_j$ be the $j$th coordinate vector in $\mathbb R^{|\mathcal J|}$. For observation~$i$ and product $j\in\mathcal J$, define $\bm a_{ij}:=(\bm e_j,\bm x_i,-p_{ij}\bm e_j)\in\mathbb R^{N_0}$, so that $u_j(\bm x_i,\bm p_i,\bm\theta)=\bm a_{ij}^{\top}\bm\theta$; set $\bm a_{i0}=\bm 0$ for the outside option. Throughout this subsection, unqualified sums over~$j$ include the outside option, and $\Delta_{\mathcal J}$ denotes the simplex defined in Section~\ref{sec:mob}. The convex conjugate of log-sum-exp gives, for any $\bm q_i\in\Delta_{\mathcal J}$,
\[
\log\sum_{j\in\{0\}\cup\mathcal J}\exp(\bm a_{ij}^{\top}\bm\theta)
\ge
\sum_{j\in\{0\}\cup\mathcal J}q_{ij}\bm a_{ij}^{\top}\bm\theta
-
\sum_{j\in\{0\}\cup\mathcal J}q_{ij}\log q_{ij}.
\]
Fenchel duality therefore supplies a family of affine lower supports indexed by auxiliary probability vectors~$\bm q_i$. For a subset $\mathcal D_s$ with index set~$\mathcal S$, define the associated score residual and entropy
\[
r_{\mathcal S}(\bm q)
=
\sum_{i\in\mathcal S}\sum_{j\in\{0\}\cup\mathcal J}
\bigl(q_{ij}-\mathbb I\{y_i=j\}\bigr)\bm a_{ij},
\qquad
\mathcal H_{\mathcal S}(\bm q)
=
-\sum_{i\in\mathcal S}\sum_{j\in\{0\}\cup\mathcal J}q_{ij}\log q_{ij}.
\]
We use the convention $0\log0=0$. Write
$\mathcal L_{\mathcal S}(\bm\theta):=\sum_{i\in\mathcal S}\ell_i(\bm\theta)$
and
$\bm\pi(\bm\theta):=\{\bm\pi_i(\bm\theta)\}_{i\in\mathcal S}$.

\noindent
Call $(\bm q,\bm\mu)$ \emph{dual feasible} on~$\mathcal S$ if
$\bm q_i\in\Delta_{\mathcal J}$ for all $i\in\mathcal S$,
$\bm\mu\ge\bm0$, and
\[
r_{\mathcal S}(\bm q)+\bm A_\Theta^\top\bm\mu=\bm0.
\]
For later use, denote the corresponding unpenalized dual value by
\[
\mathcal B_{\mathcal S}(\bm q,\bm\mu)
:=
\mathcal H_{\mathcal S}(\bm q)-\bm b_\Theta^\top\bm\mu.
\]

\begin{lemma}[Constrained Fenchel bound]\label{lem:fenchel_leaf_bound}
Every dual-feasible pair on~$\mathcal S$ satisfies
\[
H(\mathcal D_s)
\ge
\lambda N_0+\mathcal B_{\mathcal S}(\bm q,\bm\mu).
\]
\end{lemma}

A dual-feasible pair bounds the leaf cost without requiring a completed fit.
For pruning, this information may suffice to rule out the leaf before its
optimal parameters are known.
For such a pair, the Kullback--Leibler (KL) identity and the
defining residual condition give
\begin{equation}\label{eq:constrained_gap_identity}
\mathcal L_{\mathcal S}(\bm\theta)
-\mathcal B_{\mathcal S}(\bm q,\bm\mu)
=
\sum_{i\in\mathcal S}\mathrm{KL}\!\left(\bm q_i\,\|\,\bm\pi_i(\bm\theta)\right)
+
\bm\mu^\top(\bm b_\Theta-\bm A_\Theta\bm\theta),
\qquad \bm\theta\in\Theta.
\end{equation}
Both terms are nonnegative. The gap separates disagreement with the MNL
probabilities from slack in the parameter constraints. It vanishes when the
probabilities agree and the multipliers are complementary. At an attained leaf
optimum, its MNL probabilities and KKT multipliers meet these conditions.
Consequently,
\[
H(\mathcal D_s)
=
\lambda N_0+
\max_{\substack{
\bm q_i\in\Delta_{\mathcal J},\ i\in\mathcal S;\ \bm\mu\ge\bm0\\
r_{\mathcal S}(\bm q)+\bm A_\Theta^\top\bm\mu=\bm0}}
\left\{
\mathcal H_{\mathcal S}(\bm q)-\bm b_\Theta^\top\bm\mu
\right\}.
\]
Maximizing over the dual-feasible pairs therefore recovers the exact leaf
value and requires work comparable to fitting the leaf itself.

The same Fenchel inequality also gives the standard first-order support bound.
At a leaf-fitting iterate~$\bm\theta_t$, choosing
$\bm q=\bm\pi(\bm\theta_t)$ yields
\begin{equation}\label{eq:first_order_bound}
\mathrm{FO}_t(\mathcal D_s)
:=
\lambda N_0+\mathcal L_{\mathcal S}(\bm\theta_t)
+
\inf_{\bm\theta\in\Theta}
\nabla\mathcal L_{\mathcal S}(\bm\theta_t)^{\top}
(\bm\theta-\bm\theta_t).
\end{equation}
Although exact at the leaf optimum, it requires a linear optimization over
$\Theta$ at every checkpoint. A step of constrained Newton supplies the needed
dual information without a separate linear optimization.

\begin{proposition}[Closed-form Fenchel bound]\label{prop:newton_linear_bound}
At an iterate~$\bm\theta_t$, let
$\pi_{ij}^t:=\pi_{ij}(\bm\theta_t)$,
$\bar{\bm a}_i^t:=\sum_j\pi_{ij}^t\bm a_{ij}$, and let
$(\bm d_t,\bm\mu_t)$ be a primal--dual solution of the constrained Newton
subproblem~\eqref{eq:constrained_newton_qp}. Define
\begin{equation}\label{eq:newton_linear_q}
q_{ij}^t
:=
\pi_{ij}^t
\left[
1+
(\bm a_{ij}-\bar{\bm a}_i^t)^{\top}\bm d_t
\right].
\end{equation}
If $\bm q^t\ge\bm0$, then
\[
\mathrm{LB}_t(\mathcal D_s)
:=
\lambda N_0+\mathcal H_{\mathcal S}(\bm q^t)
-\bm b_\Theta^\top\bm\mu_t
\le
H(\mathcal D_s).
\]
\end{proposition}

The vector~$\bm q^t$ is the first-order approximation to the choice probabilities
at~$\bm\theta_t+\bm d_t$. The QP stationarity condition balances its residual
with~$\bm A_\Theta^\top\bm\mu_t$, while the construction preserves each
probability sum. Nonnegativity is therefore all that remains to ensure dual
feasibility. The bound uses the same MNL context, direction, and multipliers as
the fit, with no auxiliary optimization.

Since $\pi_{ij}^t>0$,
\[
\bm q^t\ge\bm0
\quad\Longleftrightarrow\quad
1+(\bm a_{ij}-\bar{\bm a}_i^t)^{\top}\bm d_t\ge0,
\qquad \forall i,j .
\]
Let
\[
\Delta_{\mathcal S}
:=
\max_{i\in\mathcal S}
\max_{j,k\in\{0\}\cup\mathcal J}
\|\bm a_{ij}-\bm a_{ik}\|_2.
\]
Because $\bar{\bm a}_i^t$ is a convex combination of the alternative vectors,
\begin{equation}\label{eq:finite_feasibility}
\Delta_{\mathcal S}\|\bm d_t\|_2\le1
\quad\Longrightarrow\quad
\bm q^t\ge\bm0.
\end{equation}
As the Newton directions shrink, this condition becomes progressively easier
to satisfy and eventually holds if $\bm d_t\to\bm0$. At an infeasible
checkpoint, the bound is not evaluated and constrained Newton advances; no
repair problem is solved.

Conditional on $\bm q^t\ge\bm0$, combining
\eqref{eq:constrained_gap_identity} with complementarity for
\eqref{eq:constrained_newton_qp} gives the exact checkpoint gap
\begin{equation}\label{eq:pathwise_gap}
\lambda N_0+\mathcal L_{\mathcal S}(\bm\theta_t)
-\mathrm{LB}_t(\mathcal D_s)
=
\sum_{i\in\mathcal S}
\mathrm{KL}\!\left(\bm q_i^t\,\|\,\bm\pi_i^t\right)
-r_{\mathcal S}(\bm q^t)^\top\bm d_t.
\end{equation}
Because~$\bm q^t$ is the first-order probability change along~$\bm d_t$, its KL term is $\frac12\bm d_t^\top\bm G_t\bm d_t+O(\|\bm d_t\|^3)$. Hence the right-hand side of~\eqref{eq:pathwise_gap} equals
\[
-\left(
\bm g_t^\top\bm d_t+\frac12\bm d_t^\top\bm G_t\bm d_t
\right)
+O(\!\|\bm d_t\|^3),
\]
the decrease predicted by the constrained Newton model, up to third order. Thus the bound uses the same local geometry that drives exact fitting.

The next result quantifies the strength of this construction. Let
$\bm\theta_{\mathcal S}^*$ be a finite leaf optimizer and write
$\bm G^*=\nabla^2\mathcal L_{\mathcal S}(\bm\theta_{\mathcal S}^*)$. Let
$\bm A_i$ collect the row vectors~$\bm a_{ij}^\top$. For every finite
$\bm\theta$, the MNL Hessian has the same null space
$\mathcal N_{\mathcal S}:=\cap_{i\in\mathcal S}\operatorname{Null}(\bm A_i)$.
Let $\mathcal I^*$ index the constraints active at
$\bm\theta_{\mathcal S}^*$ and decompose the tangent space of this face into
its MNL-invariant and identifiable directions:
\[
\mathcal T_{\mathcal S}^*
:=
\operatorname{Null}(\bm A_{\Theta,\mathcal I^*}),
\qquad
\mathcal K_{\mathcal S}^*
:=
\mathcal T_{\mathcal S}^*\cap\mathcal N_{\mathcal S},
\qquad
\mathcal V_{\mathcal S}^*
:=
\mathcal T_{\mathcal S}^*\cap(\mathcal K_{\mathcal S}^*)^\perp.
\]
The active face is \emph{locally stable} along the selected constrained Newton
path if, whenever an iterate on this face is sufficiently close to
$\bm\theta_{\mathcal S}^*$, its selected QP solution satisfies
$\bm A_{\Theta,\mathcal I^*}\bm d_t=\bm0$ and
$\mu_{t,r}=0$ for $r\notin\mathcal I^*$.
If $\mathcal V_{\mathcal S}^*=\{\bm0\}$, the local probability model has no
identifiable tangent direction and the conclusion below is immediate. We state
the nondegenerate case.

\begin{proposition}[Convergence and local tightness]\label{prop:local_accuracy}
Suppose the active face at~$\bm\theta_{\mathcal S}^*$ is locally stable. Let
$\bm Z$ have orthonormal columns spanning $\mathcal V_{\mathcal S}^*$, and
write $\bm G(\bm\theta):=\nabla^2\mathcal L_{\mathcal S}(\bm\theta)$. On a
sufficiently small convex neighborhood~$\mathcal U$ of
$\bm\theta_{\mathcal S}^*$ within this face on which~\eqref{eq:finite_feasibility}
holds, define
\[
\begin{aligned}
m_{\mathcal S}
&:=\inf_{\bm\theta\in\mathcal U}
\lambda_{\min}\!\left(\bm Z^\top\bm G(\bm\theta)\bm Z\right),
&
L_{G,\mathcal S}
&:=\sup_{\substack{\bm\theta,\bm\theta'\in\mathcal U\\
\bm\theta\ne\bm\theta'}}
\frac{\|\bm Z^\top[\bm G(\bm\theta)-\bm G(\bm\theta')]\bm Z\|_2}
{\|\bm\theta-\bm\theta'\|_2},\\
M_{\pi,\mathcal S}
&:=\max_{i\in\mathcal S}\sup_{\bm\theta\in\mathcal U}
\|D\bm\pi_i(\bm\theta)\bm Z\|_2,
&
L_{\pi,\mathcal S}
&:=\max_{i\in\mathcal S}
\sup_{\substack{\bm\theta,\bm\theta'\in\mathcal U\\
\bm\theta\ne\bm\theta'}}
\frac{\|[D\bm\pi_i(\bm\theta)-D\bm\pi_i(\bm\theta')]\bm Z\|_2}
{\|\bm\theta-\bm\theta'\|_2},
\end{aligned}
\]
and
\[
C_{\mathcal S}
:=
\frac{L_{\pi,\mathcal S}}{2m_{\mathcal S}}
+
\frac{2M_{\pi,\mathcal S}L_{G,\mathcal S}}{m_{\mathcal S}^2},
\qquad
\pi_{\min,\mathcal S}^*
:=
\min_{i\in\mathcal S,\,j\in\{0\}\cup\mathcal J}
\pi_{ij}(\bm\theta_{\mathcal S}^*).
\]
Then $m_{\mathcal S}>0$, and every constrained Newton checkpoint
$\bm\theta_t\in\mathcal U$ satisfies, with
$\nu_t^2:=\bm d_t^\top\bm G_t\bm d_t$,
\[
\max_{i\in\mathcal S}
\|\bm q_i^t-\bm\pi_i(\bm\theta_{\mathcal S}^*)\|_2
\le C_{\mathcal S}\nu_t^2.
\]
Moreover,
\[
0\le
H(\mathcal D_s)-\mathrm{LB}_t(\mathcal D_s)
=
\sum_{i\in\mathcal S}
\mathrm{KL}\!\left(
\bm q_i^t\,\middle\|\,\bm\pi_i(\bm\theta_{\mathcal S}^*)
\right)
\le
\frac{|\mathcal S|C_{\mathcal S}^2}{\pi_{\min,\mathcal S}^*}
\nu_t^4.
\]
Consequently, $\mathrm{LB}_t(\mathcal D_s)\to H(\mathcal D_s)$ as
$\bm\theta_t\to\bm\theta_{\mathcal S}^*$ within~$\mathcal U$.
\end{proposition}

Near the leaf optimum on the locally stable face, the Newton construction leaves
a probability error of order~$\nu_t^2$. The multiplier term
in~\eqref{eq:constrained_gap_identity}
vanishes at the optimum, and KL divergence is locally quadratic in this
probability error, giving a bound gap of order~$\nu_t^4$. The gap thus closes
along the same path that completes the fit. Before this local regime is reached,
\eqref{eq:finite_feasibility} gives a sufficient condition for accepting the
bound.

Because the negative log-likelihood is nonnegative, $\lambda N_0$ is always a
valid leaf lower bound. Along a partial leaf solve, we retain the strongest
feasible bound encountered on the constrained Newton path:
\[
\underline H_t(\mathcal D_s)
=
\max\left\{
\lambda N_0,\;
\max_{0\le r\le t:\ \bm q^r\in\Delta_{\mathcal J}^{|\mathcal S|}}
\left\{
\lambda N_0+\mathcal H_{\mathcal S}(\bm q^r)
-\bm b_\Theta^\top\bm\mu_r
\right\}
\right\},
\]
which is monotone in~$t$ and always satisfies $\underline H_t(\mathcal D_s)\le H(\mathcal D_s)$.
Here the inner maximum is $-\infty$ if no feasible $\bm q^r$ has yet been generated.

At checkpoint~$t$, the algorithm compares $\underline H_t(\mathcal D_s)$ with
the incumbent cap~$U$. If the bound reaches~$U$, the current leaf call returns
$\textsf{PRUNED}$. If it
does not, the algorithm takes the already computed constrained Newton step and
continues from the resulting iterate. No context is recomputed, so the contexts
visited by the interrupted solve form a prefix of those generated by an
uninterrupted exact fit. We refer to this reuse as \emph{pathwise continuation}.

\subsubsection{Cross-Set Propagation}\label{sec:bound_prop}

Pathwise continuation avoids restarting a leaf after its first evaluation, but
it cannot remove that initial evaluation. Cross-set propagation addresses this
remaining cost by using dual-feasible pairs from subsets to bound a target leaf
before its own optimization begins. The next result gives two forms of
propagation: inheritance from one subset and aggregation across several
disjoint subsets.

\begin{proposition}[Cross-set propagation of Fenchel bounds]\label{prop:cross_set_propagation}
Let $\mathcal S$ be a target index set and
$\bm y_i=(\mathbb I\{y_i=j\})_j$.

\emph{(i) Superset inheritance.} If $\mathcal E\subseteq\mathcal S$ has a
dual-feasible pair $(\bm q^{\mathcal E},\bm\mu^{\mathcal E})$, set
$\widetilde{\bm q}_i=\bm q_i^{\mathcal E}$ for $i\in\mathcal E$ and
$\widetilde{\bm q}_i=\bm y_i$ otherwise. Then
$(\widetilde{\bm q},\bm\mu^{\mathcal E})$ is dual feasible on $\mathcal S$ and
\[
\mathcal B_{\mathcal S}(\widetilde{\bm q},\bm\mu^{\mathcal E})
=
\mathcal B_{\mathcal E}(\bm q^{\mathcal E},\bm\mu^{\mathcal E}).
\]
Hence
$H(\mathcal D_s)\ge\lambda N_0+
\mathcal B_{\mathcal E}(\bm q^{\mathcal E},\bm\mu^{\mathcal E})$.

\emph{(ii) Disjoint-set aggregation.} More generally, let $\mathcal G$ be a
collection of pairwise disjoint subsets of $\mathcal S$, each with a
dual-feasible pair $(\bm q^{\mathcal E},\bm\mu^{\mathcal E})$. Set
$\widetilde{\bm q}_i=\bm q_i^{\mathcal E}$ for $i\in\mathcal E$,
$\widetilde{\bm q}_i=\bm y_i$ for observations not covered by~$\mathcal G$,
and $\widetilde{\bm\mu}=\sum_{\mathcal E\in\mathcal G}\bm\mu^{\mathcal E}$.
Then $(\widetilde{\bm q},\widetilde{\bm\mu})$ is dual feasible on $\mathcal S$
and
\[
\mathcal B_{\mathcal S}(\widetilde{\bm q},\widetilde{\bm\mu})
=
\sum_{\mathcal E\in\mathcal G}
\mathcal B_{\mathcal E}(\bm q^{\mathcal E},\bm\mu^{\mathcal E}).
\]
Consequently,
$H(\mathcal D_s)\ge\lambda N_0+
\sum_{\mathcal E\in\mathcal G}
\mathcal B_{\mathcal E}(\bm q^{\mathcal E},\bm\mu^{\mathcal E})$.
\end{proposition}

The transfer relies on observation-level additivity. Setting
$\widetilde{\bm q}_i=\bm y_i$ on uncovered observations adds neither entropy nor
residual, so a subset's dual value remains valid on the larger set. On disjoint
subsets, the dual values and multipliers add, with the target leaf penalty
included only once. The source bounds may come from partial fits and need not
correspond to children in the current tree. Unlike memoization, which retrieves
the value of the same state, this transfer can bound a different leaf before its
first MNL evaluation.

\begin{figure}[!htbp]
\centering
\includegraphics[width=0.94\linewidth]{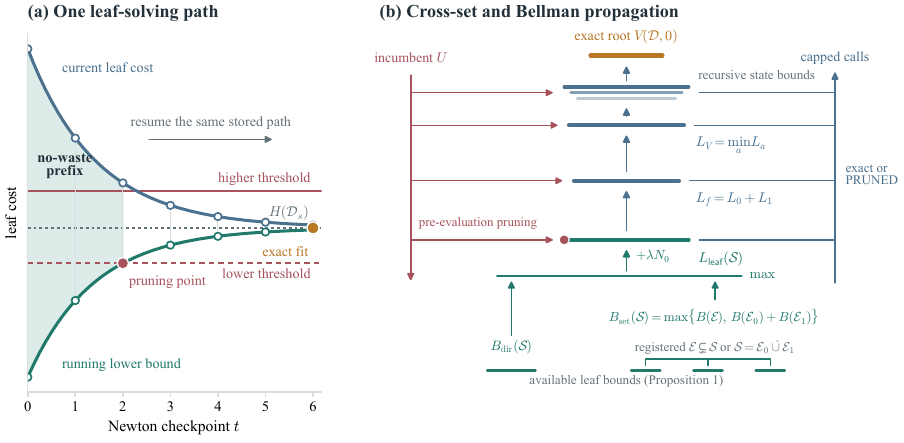}
\caption{Pruning with Fenchel bounds. (a) A direct bound check and exact fitting
share one constrained Newton path. (b) Superset inheritance and disjoint-set
aggregation provide bounds for unevaluated leaves before Bellman propagation.}
\label{fig:pathwise_bellman_pruning}
\end{figure}

To combine the available information, let $B_{\mathrm{dir}}(\mathcal S)$ be the
best unpenalized dual value generated directly on the target leaf, including
the trivial value zero. Let $\mathcal I_{\mathrm{reg}}(\mathcal S)$ contain
the registered proper subsets of~$\mathcal S$, and let
$\mathcal U_{\mathrm{reg}}(\mathcal S)$ contain the registered pairs
$(\mathcal E_0,\mathcal E_1)$ such that
$\mathcal S=\mathcal E_0\mathbin{\dot\cup}\mathcal E_1$. With a maximum over
an empty family defined as zero, the propagated and combined bounds satisfy
\[
\begin{aligned}
B_{\mathrm{set}}(\mathcal S)
&:=\max\left\{
\max_{\mathcal E\in\mathcal I_{\mathrm{reg}}(\mathcal S)}B(\mathcal E),
\max_{(\mathcal E_0,\mathcal E_1)\in\mathcal U_{\mathrm{reg}}(\mathcal S)}
\{B(\mathcal E_0)+B(\mathcal E_1)\}
\right\},\\
B(\mathcal S)
&:=\max\left\{B_{\mathrm{dir}}(\mathcal S),B_{\mathrm{set}}(\mathcal S)\right\},\\
L_{\textsf{leaf}}(\mathcal S)
&:=\lambda N_0+B(\mathcal S)
\le H(\mathcal D_s).
\end{aligned}
\]
The first inner term gives superset inheritance and the second gives disjoint-set
aggregation. Repeated local updates compose these two operations over the
registered relation graph. Section~\ref{sec:full_pruned_dp} describes how these
set bounds are combined with Bellman bounds, and
Section~\ref{sec:synth_efficiency} measures the resulting reduction in MNL
evaluations.

\subsection{The Full Pruned Dynamic Program}\label{sec:full_pruned_dp}

These leaf bounds enter the dynamic program through the incumbent cap~$U$ at
state $(\mathcal D_s,d)$.
For state $(\mathcal D_s,d)$, augment the split set in~\eqref{eq:dp} with the
stop action
$\bar{\mathcal A}(\mathcal D_s,d):=\{\textsf{leaf}\}\cup\mathcal A(\mathcal D_s,d)$.
The exact action values are $Q_{\textsf{leaf}}:=H(\mathcal D_s)$ and
$Q_f:=V(\mathcal D_s^{f=0},d+1)+V(\mathcal D_s^{f=1},d+1)$, so
$V(\mathcal D_s,d)=\min_{a\in\bar{\mathcal A}(\mathcal D_s,d)}\{Q_a\}$.
Use $L_{\textsf{leaf}}(\mathcal S)$ for the stop action and, for a split~$f$,
write $L_b:=L_V(\mathcal D_s^{f=b},d+1)$ for $b\in\{0,1\}$ and set
$L_f=L_0+L_1$. Define
$L_V:=\min_{a\in\bar{\mathcal A}(\mathcal D_s,d)}\{L_a\}$. Then
\[
L_a\ge U
\ \Longrightarrow\ a\text{ is pruned under cap }U,
\qquad
L_V\ge U
\ \Longrightarrow\ (\mathcal D_s,d)\text{ is pruned under cap }U.
\]
A capped call returns $(\textsf{EXACT},Q)$ when its exact value satisfies
$Q<U$; otherwise it returns $(\textsf{PRUNED},L)$ with $U\le L\le Q$.
A pruning return is therefore relative to its incoming cap. The bound and any
partial leaf solve remain stored, so a later call under a different cap can
continue without restarting. Applying this return rule recursively yields valid
action and state bounds at every depth.

\begin{figure}[!htbp]
\centering
\includegraphics[width=0.98\linewidth]{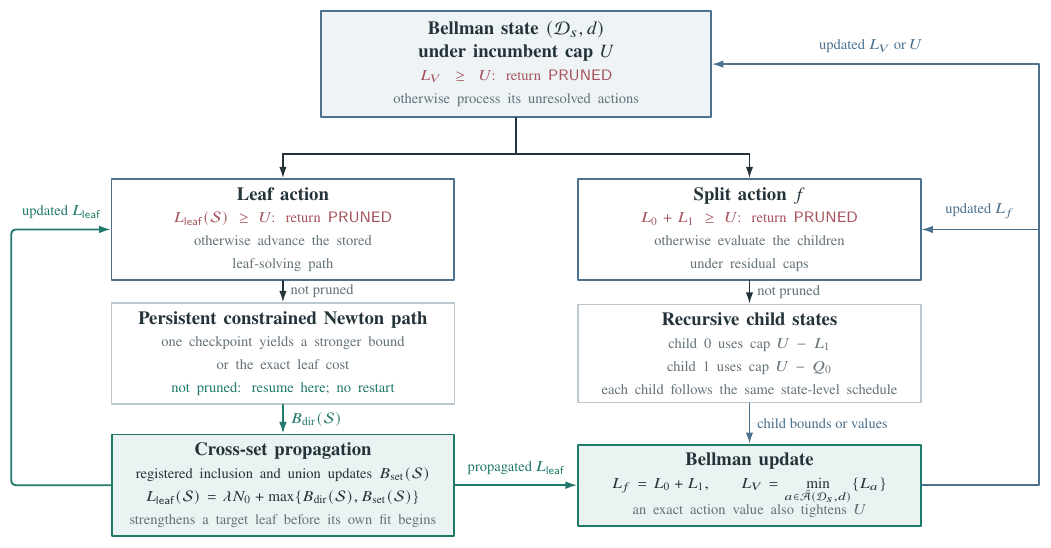}
\caption{Scheduling and information flow in the full pruned dynamic program.
Each Bellman state tests its leaf and split actions against an incumbent cap. An
unsuccessful pruning test at a leaf resumes the same constrained Newton path, while split
evaluation recurses under residual caps. Cross-set and Bellman updates return
stronger bounds and exact values to the search.}
\label{fig:full_pruned_schedule}
\end{figure}

When a state is opened, the algorithm first registers its set relations and
refreshes the available leaf, split, and state bounds. This order allows a
propagated bound to remove an action before a new MNL context is evaluated. If
$L_V\ge U$, the current call returns $\textsf{PRUNED}$. Otherwise, the leaf action and the candidate
splits are considered in a fixed order, and an exact action value below~$U$
tightens the cap for every action that remains.

For the leaf action, the algorithm first tests
$L_{\textsf{leaf}}(\mathcal S)$. If the test does not prune, fitting begins at,
or resumes from, the first unevaluated checkpoint of the stored
constrained Newton path. A dual-feasible pair strengthens
$B_{\mathrm{dir}}(\mathcal S)$; at convergence, the solver returns the leaf
value. When the bound remains below the cap, the step
already computed at that checkpoint advances the fit. Thus no MNL context or
constrained Newton subproblem is repeated after an unsuccessful pruning test.

For a split~$f$, the child-state bounds first give $L_f=L_0+L_1$. If this sum
does not prune, the first child is solved under the residual cap $U-L_1$. Only
an exact return below this cap can leave the split competitive; the second
child is then solved under $U-Q_0$, where $Q_0$ is the first child's exact
value. Each child applies the same schedule recursively. This is how leaf-level
bounds propagate through the full depth of the Bellman recursion.

Every increase in a direct leaf bound, and every newly registered set relation,
triggers the cross-set update in
Proposition~\ref{prop:cross_set_propagation}, followed by the Bellman updates of
$L_f$ and~$L_V$. Each capped call therefore has one of two valid outcomes: an
exact value below the incoming cap or a lower bound that reaches it. Induction
on the remaining depth shows that the root call with cap~$+\infty$ returns
$V(\mathcal D,0)$. Propagation uses local updates over registered inclusion,
binary-union, and parent-state relations; it introduces neither an auxiliary
optimization nor a new Bellman state. Algorithms~\ref{alg:pruned_dp}--\ref{alg:propagate_bound}
and Section~\ref{app:pruned_dp_correctness} give the complete routines and
correctness argument.

\section{Optimal Pricing and Online Deployment}\label{sec:solution}
Once a customer is routed to a leaf, the fitted demand model must produce a multi-product price vector within the latency of an online offer. The common feature coefficient in each MNL leaf reduces customer-level variation to a scalar score. This structure yields a one-dimensional characterization of the bounded pricing problem, a likelihood-based bound on pricing loss, and a lookup rule that moves numerical optimization offline.

\subsection{Within-Leaf Pricing}\label{sec:pricing}
Fix a leaf~$k$ and suppress tree routing and the fixed leaf parameters. Assume strictly positive price sensitivities, $\gamma_{jk}\ge \underline\gamma>0$, and finite bounds $0\le \underline p_j<\bar p_j<\infty$. The common feature coefficient~$\bm\beta_k$ summarizes customer features in the scalar index $s_k(\bm x):=\bm\beta_k^\top\bm x$, which shifts all product utilities equally relative to the outside option. Customers in the same leaf with the same index therefore face the same pricing problem. Write $A_{jk}(s):=\exp\{\alpha_{jk}+s\}$ for product attractiveness.
For prices~$\bm p$, write
\[
\pi_{jk}(\bm p,s)
=
\frac{A_{jk}(s)e^{-\gamma_{jk}p_j}}
{1+\sum_{r\in\mathcal J}A_{rk}(s)e^{-\gamma_{rk}p_r}},
\qquad
R_k(\bm p,s)=\sum_{j\in\mathcal J}p_j\pi_{jk}(\bm p,s).
\]
The MNL revenue gradient has the simple form
\begin{equation}\label{eq:mnl_revenue_gradient}
\frac{\partial R_k}{\partial p_j}(\bm p,s)
=
\pi_{jk}(\bm p,s)\{1-\gamma_{jk}(p_j-R_k(\bm p,s))\}.
\end{equation}
Setting \eqref{eq:mnl_revenue_gradient} to zero gives the adjusted-markup relation familiar from \citet{GallegoWang2014}. For an interior product, $p_j=1/\gamma_{jk}+R_k(\bm p,s)$. The term $1/\gamma_{jk}$ differs across products, but expected revenue is common to all of them. The products remain coupled, yet only this one scalar must be determined jointly. With price bounds, the KKT sign conditions clip the same expression at each product's limits, placing every optimum on a one-dimensional clipped-markup path.

For a candidate markup~$\mu$, define
\[
p_{jk}(\mu)
:=\operatorname{clip}\!\left(
\gamma_{jk}^{-1}+\mu,\,\underline p_j,\,\bar p_j
\right),
\qquad
\pi_{jk}(\mu,s)
:=\frac{A_{jk}(s)e^{-\gamma_{jk}p_{jk}(\mu)}}
{1+\sum_{r\in\mathcal J}A_{rk}(s)e^{-\gamma_{rk}p_{rk}(\mu)}}.
\]
The common term must equal expected revenue under the clipped prices it induces. Define $\Phi_k(\mu;s):=\sum_{j\in\mathcal J}p_{jk}(\mu)\pi_{jk}(\mu,s)-\mu$.
The next proposition shows that this condition is sufficient and uniquely
determines the optimal price vector.

\begin{proposition}[Within-leaf pricing]\label{prop:pricing_structure}
Fix a leaf~$k$ satisfying the positivity and boundedness conditions above. For any scalar index~$s$, the pricing problem~\eqref{obj:revenue} under the price bounds has a unique optimizer, denoted by $\bm p_k^*(s)$. It is characterized as follows.
\begin{enumerate}
\item[\textup{(i)}] \textup{Scalar fixed point.} The equation
\begin{equation}\label{eq:pricing_fixed_point}
\Phi_k(\mu;s)=0
\end{equation}
has a unique solution $\mu_k^*(s)\in[0,\bar p_{\max}]$, where $\bar p_{\max}:=\max_{j\in\mathcal J}\{\bar p_j\}$. The optimizer is $\bm p_k^*(s)=\bigl(p_{jk}(\mu_k^*(s))\bigr)_{j\in\mathcal J}$. Thus a within-leaf pricing instance is solved by bisection on $[0,\bar p_{\max}]$; each evaluation of~$\Phi_k$ costs $O(|\mathcal J|)$.

\item[\textup{(ii)}] \textup{Scalar sufficiency.} If two feature vectors routed to leaf~$k$ have the same value of $s_k(\bm x)$, then they have the same optimal price vector. Thus the within-leaf price map can be written as $\bm p_k^*(s)$.

\item[\textup{(iii)}] \textup{Regularity.} The map $s\mapsto \bm p_k^*(s)$ is componentwise nondecreasing and Lipschitz:
\[
\|\bm p_k^*(s)-\bm p_k^*(s')\|_\infty
\le
\frac{\bar p_{\max}}{4}|s-s'|,
\qquad
s,s'\in\mathbb R.
\]
\end{enumerate}
\end{proposition}

Appendix~\ref{app:pricing} proves the fixed-point characterization and the
regularity properties in Proposition~\ref{prop:pricing_structure}.

\paragraph{From likelihood fit to pricing loss.}
We next connect the likelihood criterion used for estimation to the revenue
criterion used for pricing. Fix a region~$\mathcal R_k$ with positive
probability whose conditional choice distribution is correctly specified by
$\bm\theta_k^0\in\Theta$, and assume
$\mathbb E[\|\bm P\|_2^2\mid\bm X\in\mathcal R_k]<\infty$. Define the
population negative log-likelihood
\[
\overline{\mathcal L}_k(\bm\theta)
=
\mathbb E\!\left[
-\log\pi_{Yk}(\bm X,\bm P;\bm\theta)
\,\middle|\,
\bm X\in\mathcal R_k
\right].
\]
We now restore the full dependence on~$(\bm x,\bm\theta)$. For arbitrary
$\bm\theta\in\Theta$, write
\[
R_k(\bm p,\bm x;\bm\theta)
=
\sum_{j\in\mathcal J}p_j\pi_{jk}(\bm x,\bm p;\bm\theta),
\qquad
\bm p_k^*(\bm x;\bm\theta)
=
\argmax_{\bm p\in\mathcal P}
R_k(\bm p,\bm x;\bm\theta).
\]
Let $T_\Theta(\bm\theta_k^0)$ be the tangent cone of~$\Theta$ at
$\bm\theta_k^0$, and define the minimum Fisher curvature over feasible
directions by
\begin{equation}\label{eq:restricted_fisher}
m_k
=
\inf_{\substack{
\bm d\in T_\Theta(\bm\theta_k^0)\\
\|\bm d\|_2=1}}
\bm d^\top
\nabla^2\overline{\mathcal L}_k(\bm\theta_k^0)
\bm d.
\end{equation}
We say that the leaf is locally identified if $m_k>0$.

Proposition~\ref{prop:pricing_structure} writes the optimal price as the
coordinatewise projection of the latent-markup vector
\[
w_{jk}(\bm x;\bm\theta):=1/\gamma_j+R_k\{\bm p_k^*(\bm x;\bm\theta),\bm x;\bm\theta\},
\]
where~$\gamma_j$ is the $j$th price-sensitivity component of~$\bm\theta$.
Its fixed-point equation implies that, on some neighborhood~$U_k^w$ of
$\bm\theta_k^0$, this vector is uniformly Lipschitz in the parameter. Let
$L_{w,k}<\infty$ satisfy
\begin{equation}\label{eq:latent_markup_lipschitz}
\sup_{\bm x\in\mathcal X}
\|\bm w_k(\bm x;\bm\theta)
-\bm w_k(\bm x;\bm\theta_k^0)\|_2
\le
L_{w,k}\|\bm\theta-\bm\theta_k^0\|_2,
\qquad
\bm\theta\in\Theta\cap U_k^w.
\end{equation}
Finally, let
\[
\bar\gamma_k=\max_{j\in\mathcal J}\{\gamma_{jk}^0\},
\qquad
M_{R,k}
=
\sup_{\substack{\bm x\in\mathcal X\\\bm p\in\mathcal P}}
\left\|
\nabla_{\bm p\bm p}^2
R_k(\bm p,\bm x;\bm\theta_k^0)
\right\|_{\mathrm{op}}.
\]

For $\bm\theta\in\Theta$, define the excess population negative log-likelihood
and the true pricing loss induced by its plug-in price as
\[
\begin{aligned}
\Delta_k^{\mathcal L}(\bm\theta)
&:=\overline{\mathcal L}_k(\bm\theta)
-\overline{\mathcal L}_k(\bm\theta_k^0),\\
\Delta_k^R(\bm\theta,\bm x)
&:=R_k\{\bm p_k^*(\bm x;\bm\theta_k^0),\bm x;\bm\theta_k^0\}
-R_k\{\bm p_k^*(\bm x;\bm\theta),\bm x;\bm\theta_k^0\}.
\end{aligned}
\]

\begin{theorem}[Likelihood as a pricing surrogate]\label{thm:likelihood_pricing}
Suppose leaf~$k$ is correctly specified, locally identified, and satisfies the
moment condition above. Then there is a neighborhood~$U_k$ of
$\bm\theta_k^0$ such that, for every $\bm\theta\in\Theta\cap U_k$ and
$\bm x\in\mathcal X\cap\mathcal R_k$,
\begin{equation}\label{eq:likelihood_pricing_constant}
0\le \Delta_k^R(\bm\theta,\bm x)
\le C_k\Delta_k^{\mathcal L}(\bm\theta),
\qquad
C_k:=\frac{4L_{w,k}^2}{m_k}
\left(\bar\gamma_k+\frac{M_{R,k}}{2}\right).
\end{equation}
If, in addition, $\Theta$ is compact, this bound holds on all of~$\Theta$
with a finite constant~$C_k^{\mathrm{glob}}$.
\end{theorem}

The constant explains how estimation error affects pricing. Larger~$m_k$
means that a given excess population likelihood loss allows less parameter
error, while~$L_{w,k}$ measures how much that error can change the latent markups.
The clipped-markup structure and revenue curvature then control the induced
pricing loss, even when the set of binding price bounds changes. Within the
theorem's local setting, a smaller excess population likelihood loss therefore
gives a tighter upper bound on pricing loss.
Tree construction remains distinct because it determines which observations
share a leaf. Appendix~\ref{app:likelihood_pricing} gives the proof and the
compact-$\Theta$ extension.

\subsection{Online Deployment}\label{sec:lookup}
Although scalar bisection is inexpensive, precomputation avoids online
optimization altogether. For each leaf, we precompute the map from scalar
index to optimal prices. Let
\[
\mathcal X_k:=\mathcal X\cap\mathcal R_k(\mathbb{T}^*),
\qquad
\underline s_k:=\inf_{\bm x\in\mathcal X_k}\{\bm\beta_k^\top\bm x\},
\qquad
\bar s_k:=\sup_{\bm x\in\mathcal X_k}\{\bm\beta_k^\top\bm x\}.
\]
These endpoints are finite because~$\mathcal X$ is compact; they are obtained from the leaf routing constraints and feature bounds. The deployed policy is tree routing followed by a table query, as illustrated in Figure~\ref{fig:lookup_pipeline}.

\paragraph{Precomputation.}
For each leaf~$k$, discretize $[\underline s_k,\bar s_k]$ into $M$ equally spaced grid points. At each grid point, solve~\eqref{eq:pricing_fixed_point} and store the resulting table $\{(s_k^{(m)},\bm p_k^*(s_k^{(m)})):m=1,\ldots,M\}$.
The one-time cost is $O(K(\mathbb{T}^*)M\mathcal T_{\textup{fp}})$, where $\mathcal T_{\textup{fp}}$ is the cost of one scalar fixed-point solve.

\paragraph{Price evaluation.}
Given~$\bm x$, the system traverses $\mathbb{T}^*$ to its leaf~$k$, computes $s_k(\bm x)=\bm\beta_k^\top\bm x$, and retrieves the price vector at the nearest grid point. Equal spacing makes the table index a constant-time calculation. Because all fixed-point solves occur offline, per-query complexity is $O(D+|\mathcal F|)$, independent of~$M$ and $\mathcal T_{\textup{fp}}$.

The next proposition bounds the lookup error; Appendix~\ref{app:lookup} gives the proof.

\begin{proposition}[Lookup approximation]\label{prop:lookup}
Under the conditions of Proposition~\ref{prop:pricing_structure}, fix a leaf~$k$ and let $\Delta_k=(\bar s_k-\underline s_k)/(M-1)$ be the spacing of an equally spaced grid with $M\ge2$ points. Let $s_m$ be the grid point nearest to~$s$, write $\tilde{\bm p}_k(s):=\bm p_k^*(s_m)$, and set $\gamma_{\max}:=\max_{j\in\mathcal J}\{\gamma_{jk}\}$. For all $s\in[\underline s_k,\bar s_k]$, the nearest-neighbor table satisfies
\begin{enumerate}
\item[\textup{(i)}] \textup{(Price error)} $\displaystyle\|\bm p_k^*(s)-\tilde{\bm p}_k(s)\|_\infty \le \frac{\bar p_{\max}}{8}\Delta_k = O(1/M)$.
\item[\textup{(ii)}] \textup{(Revenue loss)}
\[
R_k\bigl(\bm p_k^*(s),s\bigr)
-
R_k\bigl(\tilde{\bm p}_k(s),s\bigr)
\le
\frac{3\gamma_{\max}(2+\gamma_{\max}\bar p_{\max})\bar p_{\max}^2}{128}
\Delta_k^2
=O(1/M^2),
\]
\end{enumerate}
\end{proposition}

Lipschitz continuity gives the linear price-error bound. For revenue, interior
optimal prices have zero marginal revenue, while a price that remains at the
same bound does not move. If a price leaves its bound at the neighboring grid
point, its displacement and its marginal revenue at the original optimum are
both bounded by multiples of the grid spacing. The revenue loss is therefore
quadratic even when binding bounds change. Halving the spacing halves the
price-error bound and quarters the revenue-loss bound, relative to exact pricing
under the same fitted model.

\begin{figure}[htb]
\centering
\includegraphics[width=0.94\linewidth]{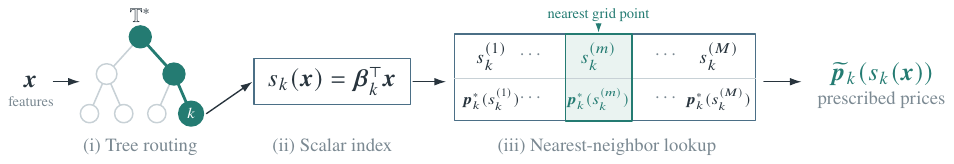}
\caption{Leaf-indexed lookup for price evaluation: a feature vector~$\bm{x}$ is routed to a leaf~$k$, its scalar index $s_k(\bm{x})=\bm{\beta}_k^\top\bm{x}$ is computed, and the prescribed price vector~$\tilde{\bm p}_k(s_k(\bm{x}))$ is retrieved from the leaf's precomputed one-dimensional table by nearest-neighbor matching.\label{fig:lookup_pipeline}}
\end{figure}

\section{Synthetic Experiments}\label{sec:synthetic}
The synthetic experiments test whether pruning makes global tree estimation computationally practical and whether global construction uses a limited depth more effectively than greedy splitting. We first use an ablation to identify the computational contribution of path reuse and cross-set propagation. We then vary the conflict between the best one-step split and the best depth-limited tree, comparing predictive fit, structural recovery, and pricing performance.

\subsection{Experiment Design}\label{sec:synth_dgp}
Each synthetic sample contains $N=5{,}000$ observations, split 80\% for training and 20\% for testing. We generate five continuous features $\{x_1,x_2,x_3,x_4,x_5\}$ and four binary features $\{d_1,d_2,d_3,d_4\}$, of which only $x_1,x_2,x_3$, and~$d_1$ enter the true tree. Every instance has the same asymmetric five-leaf structure rooted at~$x_1$. Conditional on the leaf, choices follow a two-product MNL with segment-specific intercepts and price sensitivities.

For the model comparison, we vary the agreement between the best one-step root split and the optimal depth-limited tree while holding the target structure fixed. In the \emph{Aligned} regime, the one-step criterion also selects~$x_1$. The \emph{Transition} regime uses ten calibrated parameter settings around the point at which the population root ranking changes; five favor~$x_1$ and five favor~$d_1$. In the \emph{Conflict} regime, the one-step criterion selects~$d_1$, although this split cannot represent the right-branch heterogeneity carried by~$x_2$ and~$x_3$ within the prescribed depth. Every split in the true tree remains worthwhile after its BIC penalty in all three regimes. Appendix~\ref{app:synth_dgp} gives the full construction.

\subsection{Efficiency of Pruning}\label{sec:synth_efficiency}

The computational experiment uses the Conflict regime. All four implementations solve the same \textup{[OCMT-MNL-EST]} instance with identical parent warm starts and the constraint $\gamma_j\ge10^{-4}$. Let $N_{\textup{tr}}=0.8N$ and let $B$ be the number of split bins per continuous feature. We consider $D\in\{3,4\}$ and $B\in\{5,10,15\}$, with $\lambda=\lambda_{\textup{BIC}}(N_{\textup{tr}})$ and $N_{\min}=\max\{0.05N_{\textup{tr}},20N_0\}$. Timings use one thread on an Apple M3 Pro with 36~GB of memory.

We compare four nested implementations. Pure is the unpruned dynamic program. Direct applies the within-leaf Fenchel bound but restarts fitting after an unsuccessful check; Pathwise resumes that Newton trajectory; and Full adds the two propagation rules in Proposition~\ref{prop:cross_set_propagation}. All return the same optimum. Table~\ref{tab:synthetic_pruning_efficiency} reports time and workload, where an MNL context evaluation computes the probabilities and derivatives used for bounding or fitting, and pre-fit pruning is the share of candidate leaves removed before their first context evaluation.

\begin{table}[!t]
\TABLE
{Computational ablation across 60 paired instances ($N=5{,}000$, BIC; 10 replications per setting).\label{tab:synthetic_pruning_efficiency}}
{\begin{minipage}{\linewidth}
\textit{Panel A. Computational time (seconds)}
\par\smallskip
\begin{tabular*}{\linewidth}{@{\extracolsep{\fill}}ccrrrrr}
\toprule
$D$ & $B$ & Pure & Direct & Pathwise & Full & Full speedup \\
\midrule
3 & 5  & 7.8 (0.1)     & 2.1 (0.2)   & 2.0 (0.2)   & 3.1 (0.6)   & 2.48$\times$ \\
3 & 10 & 53.9 (0.9)    & 9.8 (0.3)   & 9.7 (0.4)   & 9.0 (0.8)   & 5.96$\times$ \\
3 & 15 & 171.3 (4.0)   & 28.7 (1.3)  & 28.5 (1.2)  & 25.2 (1.6)  & 6.79$\times$ \\
\addlinespace[2pt]
4 & 5  & 36.6 (0.4)    & 8.7 (0.3)   & 8.5 (0.3)   & 9.0 (0.7)   & 4.04$\times$ \\
4 & 10 & 502.8 (4.7)   & 100.9 (3.7) & 99.5 (3.2)  & 71.7 (2.1)  & 7.01$\times$ \\
4 & 15 & 2,227.3 (27.0)& 439.6 (12.2)& 436.2 (16.4)& 311.6 (12.9)& 7.15$\times$ \\
\midrule
\multicolumn{2}{l}{All 60 (total)} & 29,996.8 & 5,896.8 & 5,843.9 & 4,297.0 & 6.98$\times$ \\
\bottomrule
\end{tabular*}

\vspace{5pt}
\textit{Panel B. Search workload by setting}
\par\smallskip
\begin{tabular*}{\linewidth}{@{\extracolsep{\fill}}lrrrrrr}
\toprule
 & \multicolumn{3}{c}{$D=3$} & \multicolumn{3}{c}{$D=4$} \\
\cmidrule(lr){2-4}\cmidrule(lr){5-7}
 & $B=5$ & $B=10$ & $B=15$ & $B=5$ & $B=10$ & $B=15$ \\
\midrule
\multicolumn{7}{l}{\textit{Exact leaf fits}} \\
Pure     & 9,679 & 64,982 & 202,288 & 47,111 & 624,470 & 2,759,635 \\
Direct   & 28    & 18     & 41      & 54     & 57      & 240 \\
Pathwise & 28    & 18     & 41      & 54     & 57      & 242 \\
Full     & 20    & 38     & 45      & 41     & 133     & 346 \\
\addlinespace[3pt]
\multicolumn{7}{l}{\textit{MNL context evaluations}} \\
Pure     & 31,342 & 209,759 & 650,285 & 154,074 & 2,034,980 & 8,957,737 \\
Direct   & 10,475 & 65,437  & 201,150 & 52,010  & 647,196   & 2,826,392 \\
Pathwise & 10,000 & 63,313  & 195,198 & 48,865  & 615,403   & 2,701,426 \\
Full     & 9,340  & 42,397  & 119,804 & 40,019  & 307,285   & 1,150,728 \\
\addlinespace[3pt]
\multicolumn{7}{l}{\textit{Pre-fit pruning under Full (\%)}} \\
Inheritance & 2.9 & 33.8 & 40.2 & 12.6 & 48.0 & 56.1 \\
Aggregation & 6.3 & 3.8  & 2.6  & 9.8  & 5.8  & 3.6 \\
All         & 9.2 & 37.9 & 43.1 & 22.7 & 54.5 & 60.8 \\
\bottomrule
\end{tabular*}
\end{minipage}}
{Entries are means over 10 paired replications; parentheses in Panel A give standard deviations. Pre-fit pruning is the share of candidate leaves eliminated before their first MNL context evaluation.}
\end{table}

Direct reveals the strength of the within-leaf bound: it reduces MNL context evaluations by 68.4\% and nearly eliminates completed fits. By resuming unsuccessful checks, Pathwise removes a further 4.4\% of the remaining contexts. Its incremental gain is modest because most direct checks already succeed. Traversal-order differences explain the small variation in exact-fit counts among pruned variants; all avoid more than 99.98\% of Pure's fits.

Cross-set propagation supplies the main scaling gain. Full reduces context evaluations by a further 54.1\% relative to Pathwise and by 86.1\% relative to Pure. It eliminates 57.7\% of candidate leaves before their first context evaluation. Superset inheritance and disjoint-set aggregation account for 52.8 and 4.0 percentage points of this rate, respectively, and together explain 98.5\% of all pre-fit exclusions. The remaining exclusions follow from the Bellman recursion.

The value of propagation grows with the state space. At $B=5$, its fixed bookkeeping cost is not yet offset by the contexts removed, and Full is slightly slower than Pathwise. From $B=5$ to 15, its incremental context reduction rises from 6.6\% to 38.6\% for $D=3$ and from 18.1\% to 57.4\% for $D=4$. At $D=4$ and $B=15$, Full eliminates 60.8\% of candidate leaves before fitting and takes 311.6 seconds versus 2,227.3 for Pure, a 7.15$\times$ speedup. Across all 60 instances, Full reduces exact fits by 99.98\%, context evaluations by 86.13\%, and computational time by 85.68\%; the last two reductions closely track each other.

\subsection{Optimal versus Greedy Choice Model Trees}\label{sec:synth_results}

We compare a homogeneous Single MNL, GCMT-MNL \citep{Aouadetal2023}, and OCMT-MNL. The two tree methods use the same MNL leaf class but different construction and complexity selection. OCMT-MNL uses $D=4$, $B=10$, and BIC. GCMT-MNL uses the tree settings in the authors' CMT experiment script: maximum growth depth~14, minimum node weight~50, a quantile step of~0.05, and validation-based pruning. All methods use the same constrained MNL specification. Prices are optimized under each fitted model between zero and twice the base prices and evaluated under ground-truth demand.

Table~\ref{tab:synthetic_main} reports predictive, structural, and pricing performance. Let $\mathcal I_{\textup{te}}$ denote the test observations. Predictive fit is measured by out-of-sample negative log-likelihood (OOS NLL), $-\sum_{i\in\mathcal I_{\textup{te}}}\log\hat\pi_{i,y_i}$, where $\hat\pi_{i,y_i}$ is the fitted probability of the realized choice. Let $k_i^0$ denote observation~$i$'s true segment and define $R_i(\bm p):=R_{k_i^0}(\bm p,\bm x_i;\bm\theta_{k_i^0}^0)$. If $\bm p_i^*$ is the oracle price and $\widehat{\bm p}_i$ is the price prescribed by a fitted model, revenue loss is
\[
\frac{\sum_{i\in\mathcal I_{\textup{te}}}
\left\{R_i(\bm p_i^*)-R_i(\widehat{\bm p}_i)\right\}}
 {\sum_{i\in\mathcal I_{\textup{te}}}R_i(\bm p_i^*)}
\times 100\%.
\]
The root metric records how often a method selects the true root~$x_1$. Structural recovery uses the adjusted Rand index (ARI), which equals one for identical partitions and has expectation zero for independent random partitions; leaf count measures parsimony. Appendix~\ref{app:synth_dgp} defines ARI.

\begin{table}[htb]
\TABLE
{Predictive, structural, and pricing performance across three synthetic regimes ($N=5{,}000$, 10 replications per regime).\label{tab:synthetic_main}}
{\begin{tabular*}{\linewidth}{@{\extracolsep{\fill}}lccccc}
\toprule
Method & OOS NLL & Revenue loss (\%) & Root $x_1$ (\%) & ARI & Leaves \\
\midrule
\multicolumn{6}{l}{\textit{Panel A. Aligned}} \\
\addlinespace[1pt]
Single MNL & 974.8 (19.8) & 7.20 (1.63) & -- & 0.000 (0.000) & 1.0 (0.0) \\
GCMT-MNL & 818.9 (21.2) & 5.63 (3.15) & \textbf{100\%} & 0.858 (0.038) & 6.2 (0.4) \\
OCMT-MNL & \textbf{817.4 (18.3)} & \textbf{4.85 (2.80)} & \textbf{100\%} & \textbf{0.969 (0.016)} & \textbf{5.0 (0.0)} \\
\addlinespace[4pt]
\multicolumn{6}{l}{\textit{Panel B. Transition}} \\
\addlinespace[1pt]
Single MNL & 977.0 (19.6) & 6.72 (1.52) & -- & 0.000 (0.000) & 1.0 (0.0) \\
GCMT-MNL & 836.2 (20.0) & 6.66 (3.29) & 50\% & 0.715 (0.149) & 8.1 (2.3) \\
OCMT-MNL & \textbf{823.8 (17.8)} & \textbf{5.30 (3.34)} & \textbf{100\%} & \textbf{0.969 (0.016)} & \textbf{5.0 (0.0)} \\
\addlinespace[4pt]
\multicolumn{6}{l}{\textit{Panel C. Conflict}} \\
\addlinespace[1pt]
Single MNL & 1012.3 (17.5) & 4.73 (1.82) & -- & 0.000 (0.000) & 1.0 (0.0) \\
GCMT-MNL & 868.7 (17.8) & 4.69 (1.71) & 0\% & 0.709 (0.043) & 8.6 (1.3) \\
OCMT-MNL & \textbf{845.7 (17.9)} & \textbf{3.13 (2.71)} & \textbf{100\%} & \textbf{0.969 (0.016)} & \textbf{5.0 (0.0)} \\
\bottomrule
\end{tabular*}}
{Each panel contains 10 paired replications; parentheses give standard deviations. Revenue loss is evaluated under ground-truth demand after optimizing prices by~\eqref{eq:pricing_fixed_point}, with bounds from zero to twice the base prices. Bold indicates the better tree method.}
\end{table}

The Aligned regime is the natural control. Both tree methods select~$x_1$ at the root in every replication. GCMT-MNL nevertheless uses 6.2 leaves on average and attains an ARI of 0.858, whereas OCMT-MNL recovers exactly five leaves and attains an ARI of 0.969. Relative to GCMT-MNL, OCMT-MNL lowers OOS NLL by 1.48 and revenue loss by 0.78 percentage points, with 95\% paired-$t$ intervals $[-5.29,8.25]$ and $[-0.09,1.64]$. When local and global split criteria agree, the main benefit of global construction is therefore a cleaner representation of the underlying segmentation.

The distinction sharpens in the Transition regime. GCMT-MNL selects~$x_1$ in five replications and~$d_1$ in five, uses 8.1 leaves on average, and attains an ARI of 0.715. OCMT-MNL retains the true root and five-leaf structure in every replication. Relative to GCMT-MNL, it reduces OOS NLL by 12.33 (95\% interval $[0.40,24.26]$) and revenue loss by 1.35 percentage points ($[0.39,2.32]$).

In the Conflict regime, GCMT-MNL selects~$d_1$ at the root in all 10 replications and grows to 8.6 leaves. OCMT-MNL again selects the true root feature and returns five leaves, with an ARI of 0.969. Its OOS NLL advantage rises to 22.97 ($[11.19,34.74]$), and its revenue loss is lower by 1.56 percentage points ($[0.15,2.97]$). The additional greedy leaves recover part of the predictive fit lost at the root, but not the true segmentation or its downstream pricing performance.

Across the three regimes, the value of global construction grows with the conflict between a locally attractive split and the best depth-limited tree. This pattern complements Theorem~\ref{thm:likelihood_pricing}: likelihood controls revenue loss within a correctly specified leaf, while tree construction determines which observations share that leaf.

The experiments in this subsection use a known MNL tree as the ground truth. Appendix~\ref{app:airline_oos} extends the comparison to the historical data used to train the market-specific models deployed in Section~\ref{sec:field}. OCMT-MNL attains lower OOS NLL than GCMT-MNL in 35 markets and a statistically significant improvement on average.

\section{Field Experiment in Ancillary Seat Pricing}\label{sec:field}
Computational and synthetic evidence does not establish whether the complete policy improves decisions in practice. We therefore deployed market-specific OCMT-MNL pricing policies in a 23-week randomized experiment across 48 airline markets. The analysis explains how the offline models were converted into online prices, defines the experimental estimand, and then examines the revenue effect and the changes in seat purchases behind it.

\subsection{Deployment Setting}\label{sec:field_setting}
We partnered with a major low-cost carrier in the Asia-Pacific region to deploy OCMT-MNL in ancillary seat selection pricing. The carrier offers nine seat-group products, shown schematically in Figure~\ref{fig:seatmap}. Groups~1--3 are premium seats and groups~4--9 are standard seats, with static prices decreasing by group index. Passengers may purchase advance seat selection or have a seat assigned at check-in. A change in one group's fee can shift demand to another group or away from paid seat selection altogether. The pricing policy therefore sets the nine fees jointly, using the customer and trip features available at each request.

\begin{figure}[!htb]
\centering
\includegraphics[width=0.97\linewidth]{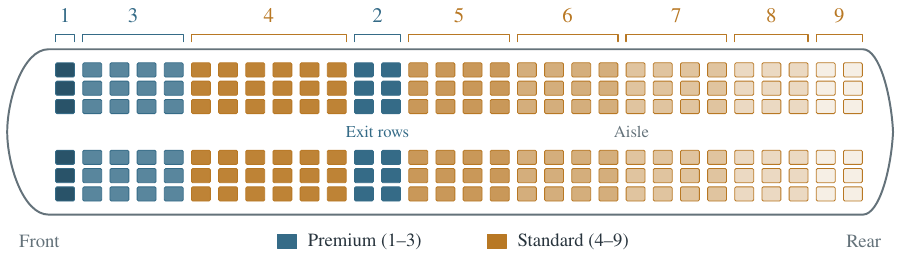}
\caption{Schematic layout of the nine seat groups. Darker shades indicate higher static prices.\label{fig:seatmap}}
\end{figure}

The carrier's Ancillary Pricing Optimization (APO) system supports real-time price experiments across digital sales channels. We train one OCMT-MNL model for each of 48 origin--destination (OD) markets using historical data from November~2024 through November~2025, before the randomized test period. These training data were collected in an earlier randomized pricing experiment. Each record includes the offered seat prices, the purchase outcome, and structural session features such as days to departure (DTD), offer stage, and the ticket fare tier; other features are withheld at the carrier's request. The offer-stage indicator distinguishes ticket booking from a post-booking revisit.\endnote{The estimation features exclude personal identifiers, device fingerprints, and cross-visit behavioral tracking.}

The deployed models use depth $D=3$, four split bins per continuous feature, the minimum leaf-size rule in Section~\ref{sec:synth_efficiency}, and $\lambda=\lambda_{\textup{AIC}}$. The depth cap reflects an operational constraint: maintaining more than eight customer segments is difficult in a live pricing system. For lookup-table construction, the nominal price bounds are 95\% and 115\% of the stage-specific static prices.\endnote{Both bounds are rounded to the nearest whole price unit before optimization, and the resulting lookup prices are rounded to the same unit. Percentage adjustments may therefore extend beyond the nominal interval.} Each request is routed through its market-specific tree and priced from the leaf-indexed table in Section~\ref{sec:lookup}; no nonlinear optimization is solved online.

The trees fitted for the OOS comparison have a median of five leaves across the 48 markets. Days to departure and offer stage account for about 40\% and 26\% of the selected splits, respectively, and offer stage is the root feature in 27 markets. These frequencies describe the learned segmentations rather than causal feature effects.

The full-data leaf models used for deployment translate customer segmentation into differentiated prices. For each market--product pair with at least three leaves and variation in both measures, we compute the rank correlation across leaves between estimated price sensitivity and the average ratio of recommended prices to static prices. Among the 360 such pairs, the correlation is negative in 98.1\% of cases and has a median of $-0.80$. The median range of the price ratio across leaves is 14.2 percentage points. Appendix~\ref{app:segmentation} reports the definitions and cross-market distributions.

Figure~\ref{fig:segment_price_profiles} illustrates the fitted model and its prices for one market. Offer stage is the root split. During ticket booking, the tree uses trip duration and then days to departure or party size. For post-booking revisits, it instead uses ticket fare and trip type. The variables that segment demand therefore differ across the two purchase stages.

Each leaf in panel~(a) reports its share of historical observations and its demand elasticities for premium and standard seats at the static prices. Panel~(b) pools the corresponding lookup-table price adjustments across observations and products within each seat group. Segments~1, 3, 4, and~6 have relatively elastic demand in at least one group and generally receive discounts. Standard demand is less elastic in segments~2 and~7, where standard seats receive sizable markups. Segment~5 recommends a markup for one premium seat product and discounts for the other two, producing the wide premium-price distribution. The group elasticities alone do not determine the price recommendations, which are chosen jointly across all nine products under the price bounds.

\begin{figure}[!htb]
\centering
\includegraphics[width=0.99\linewidth]{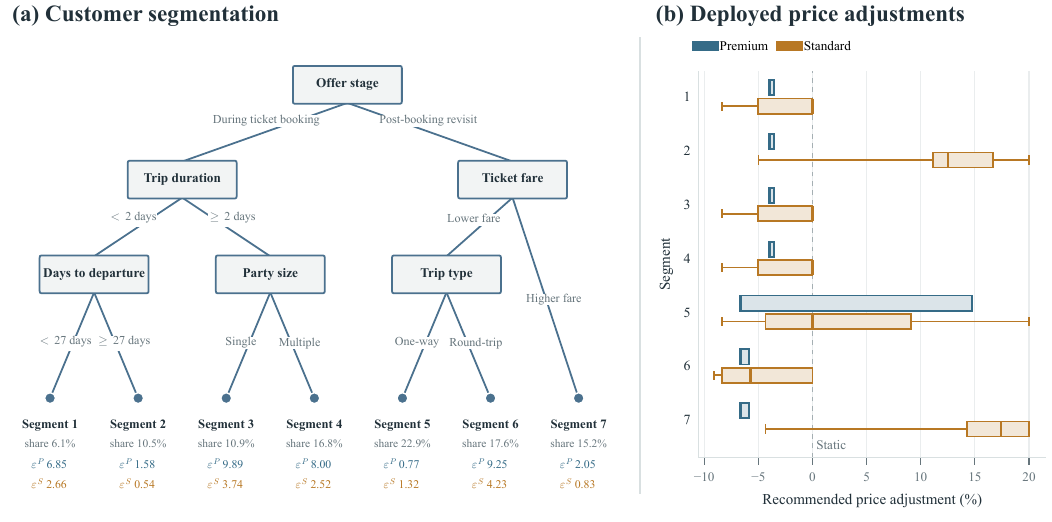}
\caption{Segmentation and deployed prices in one illustrative OD market. \textbf{(a)}~Fitted tree with segment shares and demand elasticities. \textbf{(b)}~Price adjustments across observations and seat products within each segment.\label{fig:segment_price_profiles}}
\end{figure}

\subsection{Experimental Design and Estimand}\label{sec:field_design}
From December~11, 2025 through May~20, 2026, the trained OD-specific models were evaluated in a live randomized test against the firm's static baseline prices. The APO system independently randomizes the displayed policy at each login, and the firm supplied a 10\% random sample of session logs. For each OD--booking pair, we attribute ancillary seat revenue to the policy at its earliest sampled login, the \emph{first observed assignment}. We retain pairs assigned to OCMT-MNL or the static baseline and exclude 60 pairs with conflicting policies at the same earliest timestamp. The resulting estimand is the reduced-form effect associated with that assignment, averaged over subsequent independently randomized exposures. Appendix~\ref{app:field_estimand} details the sampling mechanism.

Our primary business metric is revenue per passenger (RPP), defined as ancillary seat revenue divided by attributed passengers. For OD--booking pair~$u$, let $A_u\in\{\mathrm{OCMT\text{-}MNL},\mathrm{Static}\}$ be the first observed assignment, and let $\mathrm{Rev}_u$ and $\mathrm{Pax}_u$ be its ancillary seat revenue and passenger count. The pooled effect size is
\[
\widehat R
=
\frac{\sum_{u:A_u=\mathrm{OCMT\text{-}MNL}} \mathrm{Rev}_u}
     {\sum_{u:A_u=\mathrm{OCMT\text{-}MNL}} \mathrm{Pax}_u}
\bigg/
\frac{\sum_{u:A_u=\mathrm{Static}} \mathrm{Rev}_u}
     {\sum_{u:A_u=\mathrm{Static}} \mathrm{Pax}_u},
\]
reported after normalizing the static baseline to one. We form a 95\% confidence interval by resampling bookings within OD--assignment cells. For randomization inference, we permute assignments within OD--calendar-week strata while preserving the observed assignment counts. Market-level comparisons summarize the breadth and heterogeneity of the effect.

\subsection{Revenue Impact}\label{sec:field_results}
Over the test period, the sample attributed by first observed assignment contains 168{,}034 passengers attributed to OCMT-MNL and 22{,}186 attributed to the static baseline. Their RPP values are 1.170 and 1.051 Singapore dollars (SGD), respectively, yielding an RPP ratio of 1.113 and a lift of 11.3\%. The 95\% bootstrap interval for the lift is $[1.2\%,22.9\%]$, and the conditional randomization test gives a two-sided $p$-value of 0.027. Figure~\ref{fig:results}(a) reports the aggregate comparison.

OCMT-MNL has higher RPP in 35 of the 48 OD markets. A two-sided sign test gives $p=0.002$, and a Wilcoxon signed-rank test on the 47 finite market-level log ratios gives $p=0.016$.\endnote{The static baseline records zero ancillary revenue in one market, so its multiplicative ratio is undefined for the signed-rank test. The market remains in the sign test.} The median market lift is 25.5\%, showing that the aggregate result is broadly distributed rather than driven by one large market (Figure~\ref{fig:results}(b)).

\begin{figure}[H]
\centering
\includegraphics[width=0.84\linewidth]{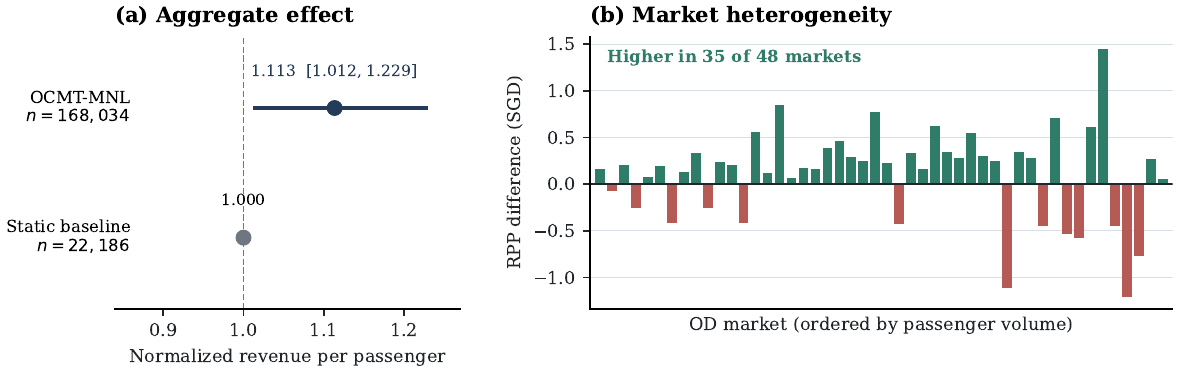}
\caption{Revenue results across the 48~OD markets. \textbf{(a)}~Normalized revenue per passenger (RPP), with the static baseline normalized to one, a 95\% bootstrap interval, and attributed passenger counts. \textbf{(b)}~Market-level RPP differences (OCMT-MNL $-$ static baseline), ordered by market volume.\label{fig:results}}
\end{figure}

\subsection{Revenue Decomposition and Premium-Seat Demand}\label{sec:field_pricing}
For assignment~$a$, let $P_a$ be the number of attributed passengers, $P_a^+$ the number in bookings with positive seat revenue, and $\mathrm{Rev}_a$ total ancillary seat revenue. The following accounting identity separates the share of passengers in positive-revenue bookings from their average realized revenue:
\[
\mathrm{RPP}_a
=
\underbrace{\frac{P_a^+}{P_a}}_{\text{positive-revenue share}}
\times
\underbrace{\frac{\mathrm{Rev}_a}{P_a^+}}_{\text{conditional revenue per passenger}}.
\]
With attribution based on the first observed assignment, the ratio of RPP for the OCMT-MNL assignment group to that for the static-baseline assignment group satisfies $1.1134=1.1074\times1.0055$.
Thus the share of attributed passengers in positive-revenue bookings is 10.7\% higher under OCMT-MNL, whereas realized revenue per passenger in those bookings differs by only 0.5\%. Most of the aggregate RPP difference is therefore associated with the first factor. The product-level counts below show that this coincides with more seats sold per passenger, especially in premium groups.

The difference is concentrated in premium seat groups. Under the static baseline, groups~1--3 account for 11.6\% of seats sold and 25.9\% of seat revenue; their realized revenue per seat sold is 2.66 times that of standard groups. Relative to the baseline, OCMT-MNL increases premium seats sold per passenger by 24.5\% and premium revenue per passenger by 27.4\%. The corresponding increases for standard groups are 5.5\% and 5.7\%. The premium share rises to 13.4\% of seats sold and 29.6\% of seat revenue, while realized revenue per premium seat sold is 2.4\% higher. The field evidence therefore points to a reallocation of purchases toward high-value seat groups, not a gain obtained by reducing their realized unit revenue.

The model diagnostics in Section~\ref{sec:field_setting} show how estimated price sensitivities are reflected in recommended prices. The experiment records more seats sold per passenger, with larger relative gains in premium groups. Randomization identifies the policy's overall effect, not separate effects of individual splits, parameters, or price changes.

\begin{figure}[htb]
\centering
\includegraphics[width=0.76\linewidth]{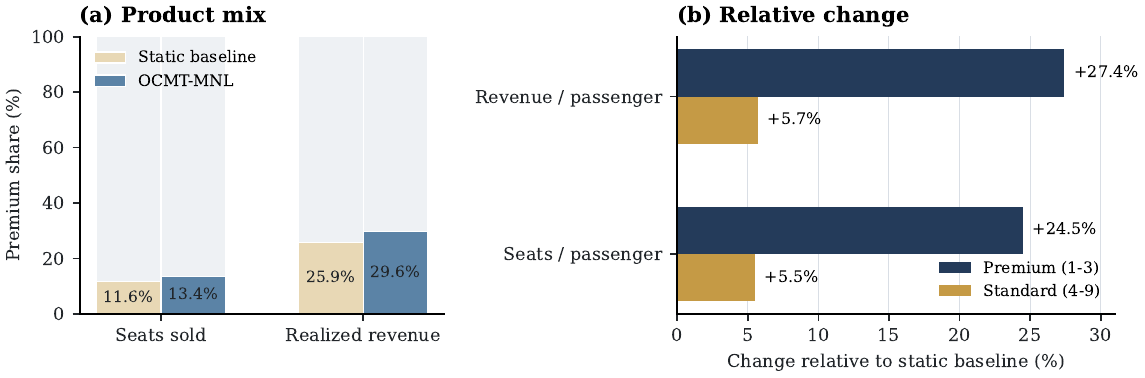}
\caption{Premium seat groups (1--3) versus standard groups (4--9). \textbf{(a)}~Shares of seats sold and realized seat revenue under the two policies. \textbf{(b)}~Relative changes in seats sold and realized revenue per passenger under OCMT-MNL.\label{fig:premium}}
\end{figure}

\section{Conclusions}\label{sec:conclusions}

This paper develops OCMT-MNL for learning compact and interpretable customer segmentations in feature-based multi-product pricing. Exact model-tree estimation is difficult because each candidate leaf requires substantial numerical optimization. Our algorithm makes leaf estimation part of the Bellman search instead of treating the fitted leaf cost as a black box. Intermediate primal-dual information yields valid lower bounds before a leaf fit is complete. After an unsuccessful pruning check, optimization resumes from the same solver state. Cross-set propagation then transfers this information across nested and disjoint datasets, allowing work on one leaf problem to prevent another leaf fit from starting.

The computational ablation confirms the importance of this interaction. Across all instances, the full algorithm eliminates 57.7\% of candidate leaves before their first MNL context evaluation, reduces exact leaf fits by 99.98\% and context evaluations by 86.13\%, and yields speedups of up to 7.15$\times$. Global construction produces more accurate, compact segmentations in synthetic experiments and improves predictive fit on real data. Separately, the randomized field experiment compares the complete OCMT-MNL pricing policy with the airline's static baseline. Across 48 markets and 190{,}220 passengers, OCMT-MNL increases seat revenue per passenger by 11.3\%. The increase is associated primarily with a higher share of passengers in positive-revenue bookings and more premium seats sold per passenger.

The present results rely on specific features of the MNL leaf model. The common feature coefficient reduces within-leaf pricing to a scalar problem. With product-specific feature effects, MNL estimation would remain convex, but pricing at deployment would generally become multidimensional. The model also assumes independence of irrelevant alternatives. Convexity alone, however, is not sufficient for our pruning approach. To apply the same architecture to another leaf model, intermediate solver iterates must provide valid lower bounds, and those bounds must remain transferable across nested and disjoint samples. We establish these properties for constrained MNL leaves. Whether richer choice models share them remains to be studied.

OCMT-MNL also clarifies why global tree construction matters. A shallow tree has only a few splits with which to represent customer heterogeneity. Greedy construction commits those splits one at a time, whereas global construction chooses how to use them jointly. Cross-set propagation makes this global search computationally practical even when every candidate leaf requires numerical estimation. The result is a tree that remains compact enough to inspect while retaining the demand structure needed for pricing.

\section*{Funding}
This research was supported by the National Research Foundation, Singapore and A*STAR, under its RIE2020 Industry Alignment Fund-Industry Collaboration Projects (IAF-ICP) Grant Call (Grant No. I2001E0059), SIA-NUS Digital Aviation Corp Lab.

\theendnotes


\clearpage
\appendix
\pdfbookmark[0]{Appendices}{appendices}
\section*{\centering\LARGE Appendices}

\section{Algorithms}\label{app:algorithms}

The Bellman recursion, leaf fitting, and split evaluation follow the same return rule under an incoming cap. Cross-set propagation updates the stored bounds used by these calls. This appendix gives the complete schedule summarized in Figure~\ref{fig:full_pruned_schedule} and the conventions needed to read its routines.

\begingroup
\raggedbottom
\setlength{\textfloatsep}{5pt}
\setlength{\floatsep}{5pt}
\setlength{\intextsep}{5pt}
\makeatletter
\newcommand{\AlgBlock}[2]{%
  \par\vspace{0.45em}%
  \hrule height 0.55pt\nobreak\vskip4pt%
  \refstepcounter{algorithm}\label{#1}%
  \noindent\textbf{Algorithm~\thealgorithm}\ #2\par
  \nobreak\vskip4pt\hrule height 0.35pt\nobreak\vskip3pt%
}
\newcommand{\EndAlgBlock}{%
  \par\vspace{0.55em}%
}
\makeatother
\subsection{Overview and Conventions}

Algorithm~\ref{alg:pruned_dp} is the incumbent-capped counterpart of
Algorithm~\ref{alg:dp_compact_final}; the remaining routines supply its leaf,
split, and propagation operations. They share $\lambda$, $\Theta$, and a fixed
action order, with the leaf action first and the splits in lexicographic order;
this order also breaks ties. For a
dataset~$\mathcal D_s$, $\mathcal S$ always denotes its observation index set.
The root call is \textsc{SolveState}$(\mathcal D,0,+\infty)$.

Three conventions govern the implementation. First, in a capped call, $U$ is
the best cost that the caller can still improve. For an object with exact value~$Q$, the call returns
$(\textsf{EXACT},Q)$ if $Q<U$; otherwise it returns
$(\textsf{PRUNED},L)$ with $U\le L\le Q$. A \textsf{PRUNED} return therefore
certifies that the object cannot improve the incumbent, while retaining the
computed lower bound for later calls. Second, the persistent store has three
parts: exact state and action values, valid action and state lower bounds, and
one resumable constrained Newton path for each index set. Third, all unopened
feasible states start from the nonnegativity bound $L_V=\lambda N_0$; an
infeasible state starts from $+\infty$. When an index set~$\mathcal S$ is first
registered, set $B_{\mathrm{dir}}(\mathcal S)=B_{\mathrm{set}}(\mathcal S)=0$
and hence $L_{\textsf{leaf}}(\mathcal S)=\lambda N_0$. Thereafter,
\[
\begin{aligned}
B(\mathcal S)&=\max\{B_{\mathrm{dir}}(\mathcal S),B_{\mathrm{set}}(\mathcal S)\},
&L_{\textsf{leaf}}(\mathcal S)&=\lambda N_0+B(\mathcal S),\\
L_f&=L_V(\mathcal D_s^{f=0},d+1)+L_V(\mathcal D_s^{f=1},d+1),
&L_V(\mathcal D_s,d)&=\min_{a\in\bar{\mathcal A}(\mathcal D_s,d)}\{L_a\}.
\end{aligned}
\]

Algorithm~\ref{alg:eval_split} uses residual caps because $Q_f=Q_0+Q_1$.
Given $L_1\le Q_1$, an improving split requires $Q_0<U-L_1$; after child~0
returns the exact value~$z_0$, it requires $Q_1<U-z_0$. Failure at either cap
therefore prunes the split. The routines interact as follows.

\begin{center}
\begin{tabularx}{0.96\linewidth}{@{}>{\bfseries}lXl@{}}
\toprule
Routine & Role & Calls \\
\midrule
\textsc{SolveState} & Minimizes over the leaf and split actions at one Bellman state. & \textsc{LeafFit}, \textsc{EvalSplit} \\
\textsc{LeafFit} & Tests the current leaf bound and, if needed, advances the same constrained Newton path. & \textsc{PropagateBound} \\
\textsc{EvalSplit} & Uses child-state bounds and capped recursive calls to evaluate one split. & \textsc{SolveState} \\
\textsc{PropagateBound} & Transfers a new direct bound through registered subset and disjoint-union relations. & --- \\
\bottomrule
\end{tabularx}
\end{center}

\medskip

A set is registered when its first Bellman state is opened. The registry stores
its index set, every verified proper inclusion with an already registered set,
and every disjoint union created by an explored split. Adding a relation places
its source set or sets in the propagation queue. The resulting dependency graph
runs from set bounds to leaf-action bounds, then from child-state bounds to
split-action bounds and parent-state bounds. Algorithm~\ref{alg:propagate_bound}
updates this graph whenever a source bound increases. State registration is
separate: whenever $(\mathcal D_s,d)$ is first opened, even if $\mathcal S$ is
already known at another depth, its leaf and split actions are initialized from
the current set and child-state bounds before any pruning test.

The pseudocode and its correctness argument are stated in exact arithmetic. In
particular, an accepted Fenchel value satisfies the QP KKT system and
$\bm q^t\in\Delta_{\mathcal J}^{|\mathcal S|}$. The numerical implementation
checks the corresponding QP residuals before using a bound.
$\mathcal A(\mathcal D_s,d)$ contains only splits whose children meet the
minimum leaf size. A leaf path starts from the fitted parent parameter when it
is available and otherwise from the fixed feasible point specified in
Section~\ref{sec:estimation}.

\subsection{Capped Bellman Search}

\AlgBlock{alg:pruned_dp}{Capped Bellman recursion}
\begin{algorithmic}[1]
\Require The model and search inputs of Algorithm~\ref{alg:dp_compact_final}
\Ensure Optimal tree $\mathbb{T}^*$ with leaf parameters $\{\bm{\theta}^*_k\}$
\Statex
\Function{SolveState}{$\mathcal{D}_s,d,U$}
\State $U_{\mathrm{in}}\leftarrow U$
\If{$|\mathcal{D}_s|<N_{\min}$} \Return $(\textsf{PRUNED},+\infty)$
\EndIf
\State Register state $(\mathcal D_s,d)$, its actions, $\mathcal S$, and the verified set relations
\State Initialize $L_{\textsf{leaf}}(\mathcal S)=\lambda N_0+B(\mathcal S)$ and each $L_f$ from its two child-state bounds
\State $L_V(\mathcal D_s,d)\leftarrow
\min_{a\in\bar{\mathcal A}(\mathcal D_s,d)}\{L_a\}$; propagate the newly activated relations
\If{the exact state value $v$ is cached}
\If{$v<U$} \Return $(\textsf{EXACT},v)$ \Else\ \Return $(\textsf{PRUNED},v)$
\EndIf
\EndIf
\If{$L_V(\mathcal D_s,d)\ge U$} \Return $(\textsf{PRUNED},L_V(\mathcal D_s,d))$
\EndIf
\State $v^*\leftarrow+\infty$;\quad $a^*\leftarrow\varnothing$
\State $(\tau,v_{\textsf{leaf}})\leftarrow\Call{LeafFit}{\mathcal{D}_s,U}$
\State $L_{\textsf{leaf}}(\mathcal S)\leftarrow
\max\{L_{\textsf{leaf}}(\mathcal S),v_{\textsf{leaf}}\}$
\State $L_V(\mathcal D_s,d)\leftarrow
\min_{a\in\bar{\mathcal A}(\mathcal D_s,d)}\{L_a\}$
\If{$\tau=\textsf{EXACT}$}
\State Cache $Q_{\textsf{leaf}}=v_{\textsf{leaf}}$;
$v^*\leftarrow v_{\textsf{leaf}}$;
$a^*\leftarrow\textsf{Leaf}$;
$U\leftarrow\min\{U,v^*\}$
\EndIf
\If{$d<D$}
\ForAll{$f\in\mathcal{A}(\mathcal{D}_s,d)$ in the prescribed order}
\If{$L_V(\mathcal D_s,d)\ge U$} \textbf{break}
\EndIf
\State $(\tau,v_f)\leftarrow\Call{EvalSplit}{\mathcal{D}_s^{f=0},\mathcal{D}_s^{f=1},d+1,U}$
\State $L_f\leftarrow\max\{L_f,v_f\}$
\If{$\tau=\textsf{EXACT}$}
\State Cache $Q_f=v_f$;\quad $L_f\leftarrow v_f$
\If{$v_f<v^*$}
\State $v^*\leftarrow v_f$;\quad $a^*\leftarrow f$;\quad $U\leftarrow v^*$
\EndIf
\EndIf
\State $L_V(\mathcal D_s,d)\leftarrow
\min_{a\in\bar{\mathcal A}(\mathcal D_s,d)}\{L_a\}$
\EndFor
\EndIf
\State $L_V(\mathcal D_s,d)\leftarrow
\min_{a\in\bar{\mathcal A}(\mathcal D_s,d)}\{L_a\}$
\If{$v^*<U_{\mathrm{in}}$}
\State Cache exact value $(v^*,a^*)$ for $(\mathcal{D}_s,d)$; \Return $(\textsf{EXACT},v^*)$
\Else
\State Cache $L_V(\mathcal D_s,d)$; \Return $(\textsf{PRUNED},L_V(\mathcal D_s,d))$
\EndIf
\EndFunction
\Statex
\State $(\textsf{EXACT},V^*)\leftarrow\Call{SolveState}{\mathcal{D},0,+\infty}$
\State Recover $(\mathbb{T}^*,\{\bm{\theta}^*_k\})$ by backtracking cached exact actions
\end{algorithmic}
\EndAlgBlock

\subsection{Leaf and Split Evaluation}

\AlgBlock{alg:progressive_leaf_fit}{Pathwise leaf evaluation}
\begin{algorithmic}[1]
\Function{LeafFit}{$\mathcal{D}_s,U$}
\If{exact leaf cost $H$ is cached}
\If{$H<U$} \Return $(\textsf{EXACT},H)$ \Else\ \Return $(\textsf{PRUNED},H)$
\EndIf
\EndIf
\State $\underline H\leftarrow L_{\textsf{leaf}}(\mathcal S)$
\If{$\underline H\ge U$} \Return $(\textsf{PRUNED},\underline H)$
\EndIf
\State Initialize the path, or resume it by taking the pending stored step
\State Let $(\bm\theta_t,t)$ be the first checkpoint not previously evaluated
\Loop
\State Compute probabilities, gradient $\bm g_t$, and Hessian $\bm G_t$ at $\bm\theta_t$
\State Solve~\eqref{eq:constrained_newton_qp} once for $(\bm d_t,\bm\mu_t)$
\If{the leaf solver has converged}
\State Cache $H(\mathcal D_s)$ and $\bm\theta^*(\mathcal D_s)$; set $B_{\mathrm{dir}}(\mathcal S)\leftarrow H(\mathcal D_s)-\lambda N_0$
\State \Call{PropagateBound}{$\{\mathcal S\}$}
\If{$H(\mathcal D_s)<U$} \Return $(\textsf{EXACT},H(\mathcal{D}_s))$ \Else\ \Return $(\textsf{PRUNED},H(\mathcal D_s))$
\EndIf
\EndIf
\State Use the damped line search of Section~\ref{sec:estimation} to obtain a feasible step $\alpha_t\bm d_t$
\State Form $\bm q^t$ using~\eqref{eq:newton_linear_q}
\If{$\bm q^t\in\Delta_{\mathcal J}^{|\mathcal S|}$}
\State $B_t\leftarrow\mathcal H_{\mathcal S}(\bm q^t)-\bm b_\Theta^\top\bm\mu_t$
\State $B_{\mathrm{dir}}(\mathcal S)\leftarrow\max\{B_{\mathrm{dir}}(\mathcal S),B_t\}$; \quad \Call{PropagateBound}{$\{\mathcal S\}$}
\State $\underline H\leftarrow L_{\textsf{leaf}}(\mathcal S)$
\If{$\underline H\ge U$}
\State Store checkpoint $t$, its MNL context, and its feasible damped step
\State \Return $(\textsf{PRUNED},\underline H)$
\EndIf
\EndIf
\State Store checkpoint $t$; set $\bm\theta_{t+1}\leftarrow\bm\theta_t+\alpha_t\bm d_t$ and $t\leftarrow t+1$
\EndLoop
\EndFunction
\end{algorithmic}
\EndAlgBlock

\AlgBlock{alg:eval_split}{Capped split evaluation}
\begin{algorithmic}[1]
\Function{EvalSplit}{$\mathcal D_0,\mathcal D_1,d,U$}
\State Register $\mathcal S_0$, $\mathcal S_1$, and
$\mathcal S=\mathcal S_0\mathbin{\dot\cup}\mathcal S_1$; propagate the new union relation
\State Retrieve valid child-state bounds $L_0$ and $L_1$
\If{$L_0+L_1\ge U$} \Return $(\textsf{PRUNED},L_0+L_1)$
\EndIf
\State $(\tau_0,z_0)\leftarrow\Call{SolveState}{\mathcal D_0,d,U-L_1}$
\If{$\tau_0=\textsf{PRUNED}$} \Return $(\textsf{PRUNED},z_0+L_1)$
\EndIf
\State $(\tau_1,z_1)\leftarrow\Call{SolveState}{\mathcal D_1,d,U-z_0}$
\If{$\tau_1=\textsf{PRUNED}$} \Return $(\textsf{PRUNED},z_0+z_1)$
\EndIf
\State \Return $(\textsf{EXACT},z_0+z_1)$
\EndFunction
\end{algorithmic}
\EndAlgBlock

\subsection{Cross-Set Bound Update}

\AlgBlock{alg:propagate_bound}{Cross-set propagation}
\begin{algorithmic}[1]
\Function{PropagateBound}{$\textsf{SetQueue}$}
\Comment{sets with a larger bound or a new outgoing relation}
\State $\textsf{StateQueue}\leftarrow\varnothing$
\While{$\textsf{SetQueue}$ is nonempty}
\State Remove a set $\mathcal E$ from $\textsf{SetQueue}$
\ForAll{registered proper inclusions $\mathcal E\subsetneq\mathcal S'$}
\If{$B(\mathcal E)>B_{\mathrm{set}}(\mathcal S')$}
\State $B_{\mathrm{set}}(\mathcal S')\leftarrow B(\mathcal E)$; add $\mathcal S'$ to $\textsf{SetQueue}$
\EndIf
\EndFor
\ForAll{registered unions $\mathcal S'=\mathcal E_0\mathbin{\dot\cup}\mathcal E_1$ affected by $\mathcal E$}
\State $c\leftarrow B(\mathcal E_0)+B(\mathcal E_1)$
\If{$c>B_{\mathrm{set}}(\mathcal S')$}
\State $B_{\mathrm{set}}(\mathcal S')\leftarrow c$; add $\mathcal S'$ to $\textsf{SetQueue}$
\EndIf
\EndFor
\ForAll{registered states $(\mathcal D_e,d)$ with index set $\mathcal E$}
\State $L_{\textsf{leaf}}(\mathcal E)\leftarrow\lambda N_0+B(\mathcal E)$; add $(\mathcal D_e,d)$ to $\textsf{StateQueue}$
\EndFor
\EndWhile
\While{$\textsf{StateQueue}$ is nonempty}
\State Remove a state $(\mathcal D_s,d)$ from $\textsf{StateQueue}$
\State Recompute each split bound $L_f$ as the sum of its two child-state bounds
\State $L'\leftarrow\min_{a\in\bar{\mathcal A}(\mathcal D_s,d)}\{L_a\}$
\If{$L'>L_V(\mathcal D_s,d)$}
\State $L_V(\mathcal D_s,d)\leftarrow L'$; add every registered parent state using this child to $\textsf{StateQueue}$
\EndIf
\EndWhile
\EndFunction
\end{algorithmic}
\EndAlgBlock

Algorithm~\ref{alg:progressive_leaf_fit} uses one constrained Newton QP
solution for the convergence test, the direct Fenchel bound, and the next SQP
step. Its path record contains the current iterate, the evaluated MNL context,
the QP direction and multiplier, and any pending damped step; resumption applies
that step once and begins at the first unevaluated checkpoint.
Algorithm~\ref{alg:propagate_bound} then applies
Proposition~\ref{prop:cross_set_propagation} to the relations registered by the
Bellman search. Consequently, a failed leaf check advances the exact fit
without recomputing a context, while a propagated bound can remove a different
leaf before its first context evaluation.

\endgroup

\section{Proofs}\label{app:proofs}

The arguments proceed from estimation to deployment. We first establish the exact Bellman recursion and the leaf-level Fenchel bounds, then show how these bounds propagate without invalidating the search. The remaining proofs characterize within-leaf pricing, likelihood-based pricing loss, and the lookup approximation.

\subsection{Proof of Theorem~\ref{thm:dp}}\label{app:dp}

\begin{proof}[Proof of Theorem~\ref{thm:dp}]

\textbf{Optimality.}
We prove by induction on $D - d$ that $V(\mathcal{D}_s, d)$ equals the optimal
objective of \textup{[OCMT-MNL-EST]} restricted to $\mathcal{D}_s$ and
$\mathfrak{T}(\mathcal{F}', D-d, N_{\min})$. At $d = D$ (base case), the only
feasible tree is a single leaf with cost $H(\mathcal{D}_s)$, which is evaluated
at its global minimum; if $|\mathcal{D}_s| < N_{\min}$, both the algorithm and the
problem assign $+\infty$. For $d < D$ (inductive step), any feasible tree either
(a)~is a single leaf with cost $H(\mathcal{D}_s)$, or (b)~splits on some
$f \in \mathcal{F}'$, partitioning $\mathcal{D}_s$ into
$\mathcal{D}_s^{f=0}$ and $\mathcal{D}_s^{f=1}$. Because the loss decomposes
additively across disjoint subtrees,
\[
L(\mathbb{T}, \{\bm{\theta}_k\} \mid \mathcal{D}_s)
=
L(\mathbb{T}_0, \{\bm{\theta}_k\}_{k\in\mathcal{K}_0} \mid \mathcal{D}_s^{f=0})
+
L(\mathbb{T}_1, \{\bm{\theta}_k\}_{k\in\mathcal{K}_1} \mid \mathcal{D}_s^{f=1}),
\]
and the inductive hypothesis ensures each subtree value is optimal. The Bellman
recursion exhaustively evaluates every $f \in \mathcal A(\mathcal D_s,d)$ and selects the
minimum over leaf and split actions, so $V(\mathcal{D}_s, d)$ is globally
optimal for state $(\mathcal{D}_s, d)$. By induction, $V(\mathcal{D}, 0)$ is
optimal over $\mathfrak{T}(\mathcal{F}', D, N_{\min})$, and backtracking
recovers $\mathbb{T}^*$ and $\{\bm{\theta}^*_k\}$. Memoization ensures each
state is solved exactly once without affecting optimality, since the value
depends only on $(\mathcal{D}_s, d)$, not on the path by which it was reached.

\textbf{Complexity.}
We bound the total running time by counting the number of distinct states
visited and the per-state computation cost.

\textbf{Counting distinct states.}
A state at depth $d$ is determined by the pair $(\mathcal{D}_s, d)$, where the
subset $\mathcal{D}_s$ is the set of observations whose binary feature
vector~$\bm{x}'$ satisfies a conjunction of $d$ conditions of the form
$x'_f = b_f$ for distinct features $f$ drawn from $\mathcal{F}'$ and bits
$b_f \in \{0,1\}$. The features along any root-to-node path are necessarily
distinct: since all elements of $\mathcal{F}'$ are binary, splitting on a
feature~$f$ that already appears on the ancestor path produces one child
identical to the parent and one empty child, which violates the
$N_{\min}\ge 1$ feasibility requirement. Furthermore, different orderings of the
same $d$~splits yield the same data subset; hence the subset is determined by
the choice of $d$~features and their values. The number of distinct subsets
at depth~$d$ is therefore at most
\[
\binom{|\mathcal{F}'|}{d} \cdot 2^d.
\]
Summing over all depths $d = 0, 1, \ldots, D$, the total number of distinct
states is at most
\[
\mathcal{N}_{\textup{DP}}
\;:=\;
\sum_{d=0}^{D} \binom{|\mathcal{F}'|}{d}\, 2^d.
\]
The constraint $D \le |\mathcal{F}'|$ is without loss of generality, since the
features along any root-to-leaf path are distinct (as argued above).

\textbf{Per-state computation.}
At each state $(\mathcal{D}_s, d)$, the algorithm performs two operations:
\begin{enumerate}
\item[(i)] \textbf{Leaf cost evaluation}: solving one segment-level MNL estimation problem on
$\mathcal{D}_s$ to obtain $H(\mathcal{D}_s)$ and $\bm{\theta}^*(\mathcal{D}_s)$.
For a state containing $n$ observations, each constrained Newton iteration
evaluates the gradient and Hessian in $O(nN_0^2)$ time, processes the
$m_\Theta$ linear constraints in $O(m_\Theta N_0)$ time, and solves one
convex QP. At tolerance~$\varepsilon$, the leaf cost is therefore
bounded by
\[
O\!\left(
\overline m_{\textup{Newt}}(N,\varepsilon)
\left\{NN_0^2+m_\Theta N_0+
\mathcal T_{\textup{QP}}(N_0,m_\Theta,\varepsilon)\right\}
\right).
\]
\item[(ii)] \textbf{Split enumeration}: for each candidate feature
$f \in \mathcal{F}'$, partitioning $\mathcal{D}_s$ into
$\mathcal{D}_s^{f=0}$ and $\mathcal{D}_s^{f=1}$. Each partition requires a
single pass over $\mathcal{D}_s$, costing $O(|\mathcal{D}_s|)$. Enumerating all
$|\mathcal{F}'|$~features costs $O(|\mathcal{F}'| \cdot |\mathcal{D}_s|) \le
O(|\mathcal{F}'| \cdot N)$.
\end{enumerate}

\textbf{Total complexity.}
Multiplying the state count by the per-state cost yields
\[
O\!\left(
\left\{\sum_{d=0}^{D}\binom{|\mathcal F'|}{d}2^d\right\}
\left[
\overline m_{\textup{Newt}}(N,\varepsilon)
\left\{NN_0^2+m_\Theta N_0+
\mathcal T_{\textup{QP}}(N_0,m_\Theta,\varepsilon)\right\}
+|\mathcal F'|N
\right]
\right).
\]
This proves the complexity statement in the theorem.

\textbf{Dependence on $N$.}
We emphasize that the sample size $N$ enters the complexity through the
per-state operations, not through the number of states. Although the set of
candidate binary features $\mathcal{F}'$ is fixed a priori (via a predetermined
discretization scheme), each state requires touching the data to (i)~fit an MNL
leaf model, and (ii)~partition the observations
for each candidate split.

\textbf{Two perspectives on tractability.}
For fixed $D$ and $|\mathcal{F}'|$, the state-count factor
$\sum_{d=0}^{D}\binom{|\mathcal{F}'|}{d}\,2^{d}$ is a constant, so the
dependence on sample size is governed by constrained leaf fitting and split
enumeration. This is the practically relevant regime, since $D$ is a small
modeler-chosen constant (typically 3--5) and $|\mathcal{F}'|$ is determined by
feature engineering and discretization. If $D$ grows with the input, the
state-count factor is exponential in~$D$.
\end{proof}

\subsection{Proof of Lemma~\ref{lem:fenchel_leaf_bound}}\label{app:fenchel_leaf}

\begin{proof}[Proof of Lemma~\ref{lem:fenchel_leaf_bound}]
For any probability vector $\bm q_i$, Fenchel--Young for log-sum-exp gives
\[
\log\sum_j\exp(\bm a_{ij}^{\top}\bm\theta)
\ge
\sum_j q_{ij}\bm a_{ij}^{\top}\bm\theta
-
\sum_j q_{ij}\log q_{ij}.
\]
Subtracting $\bm a_{iy_i}^{\top}\bm\theta$ and summing over $i\in\mathcal S$
yields
\[
\mathcal L_{\mathcal S}(\bm\theta)
\ge
\mathcal H_{\mathcal S}(\bm q)+r_{\mathcal S}(\bm q)^{\top}\bm\theta .
\]
For a dual-feasible pair,
$r_{\mathcal S}(\bm q)=-\bm A_\Theta^\top\bm\mu$. Hence every
$\bm\theta\in\Theta$ satisfies
\[
r_{\mathcal S}(\bm q)^\top\bm\theta
=-\bm\mu^\top\bm A_\Theta\bm\theta
\ge-\bm\mu^\top\bm b_\Theta.
\]
Taking the infimum over~$\Theta$ and adding $\lambda N_0$ gives
$H(\mathcal D_s)\ge\lambda N_0+\mathcal B_{\mathcal S}(\bm q,\bm\mu)$.
\end{proof}

\subsection{Proof of Proposition~\ref{prop:newton_linear_bound}}\label{app:closed_form_leaf_bound}

\begin{proof}[Proof of Proposition~\ref{prop:newton_linear_bound}]
Let
\[
\bm S_i^t
:=
\operatorname{Diag}(\bm\pi_i^t)-\bm\pi_i^t(\bm\pi_i^t)^\top,
\]
and let $\bm A_i$ collect the row vectors $\bm a_{ij}^\top$. Then
$\bm G_t=\sum_{i\in\mathcal S}\bm A_i^\top\bm S_i^t\bm A_i$.
Write $\bm\delta_i^t=\bm q_i^t-\bm\pi_i^t$. From
\eqref{eq:newton_linear_q},
\[
\bm\delta_i^t=\bm S_i^t\bm A_i\bm d_t.
\]
Because $\bm 1^\top\bm S_i^t=\bm0$, we have
$\sum_jq_{ij}^t=\sum_j\pi_{ij}^t=1$. Moreover,
\[
\begin{aligned}
r_{\mathcal S}(\bm q^t)
&=
\bm g_t+
\sum_{i\in\mathcal S}\bm A_i^\top\bm\delta_i^t \\
&=
\bm g_t+\bm G_t\bm d_t
=
-\bm A_\Theta^\top\bm\mu_t,
\end{aligned}
\]
where the last equality is the stationarity condition for
\eqref{eq:constrained_newton_qp}. For every $\bm\theta\in\Theta$,
\[
r_{\mathcal S}(\bm q^t)^\top\bm\theta
=
-\bm\mu_t^\top\bm A_\Theta\bm\theta
\ge
-\bm\mu_t^\top\bm b_\Theta.
\]
At the feasible QP endpoint $\widehat{\bm\theta}_t:=\bm\theta_t+\bm d_t$,
complementarity gives equality. Hence
\[
\inf_{\bm\theta\in\Theta}
r_{\mathcal S}(\bm q^t)^\top\bm\theta
=
-\bm b_\Theta^\top\bm\mu_t.
\]
When $\bm q^t$ is componentwise nonnegative, its components therefore form
probability vectors and Lemma~\ref{lem:fenchel_leaf_bound} yields the stated
lower bound.
\end{proof}

\subsection{Proof of Proposition~\ref{prop:local_accuracy}}\label{app:local_accuracy}

\begin{proof}[Proof of Proposition~\ref{prop:local_accuracy}]
Let $\bm\pi_i^*=\bm\pi_i(\bm\theta_{\mathcal S}^*)$ and
$\bm e_t=\bm\theta_t-\bm\theta_{\mathcal S}^*$. For any finite
$\bm\theta$, let
$\bm S_i(\bm\theta):=\operatorname{Diag}\{\bm\pi_i(\bm\theta)\}
-\bm\pi_i(\bm\theta)\bm\pi_i(\bm\theta)^\top$. The MNL Hessian is
\[
\bm G(\bm\theta)
=
\sum_{i\in\mathcal S}
\bm A_i^\top\bm S_i(\bm\theta)\bm A_i,
\]
and hence, for every direction~$\bm v$,
\[
\bm v^\top\bm G(\bm\theta)\bm v
=
\sum_{i\in\mathcal S}
\operatorname{Var}_{\bm\pi_i(\bm\theta)}
\!\left(\bm a_{ij}^\top\bm v\right).
\]
Every MNL probability is positive at finite~$\bm\theta$. The variance for
observation~$i$ is therefore zero if and only if
$\bm a_{ij}^\top\bm v$ is constant across alternatives. Because the
outside-option row is~$\bm a_{i0}=\bm0$, that constant must be zero. Thus
\begin{equation}\label{eq:fixed_mnl_null_space}
\operatorname{Null}\{\bm G(\bm\theta)\}
=
\bigcap_{i\in\mathcal S}\operatorname{Null}(\bm A_i)
=
\mathcal N_{\mathcal S}
\qquad\text{for every finite }\bm\theta.
\end{equation}
In particular, the Hessian annihilates~$\mathcal N_{\mathcal S}$ and, by
symmetry, does not couple it with identifiable directions.

On the stable face, $\bm e_t,\bm d_t\in\mathcal T_{\mathcal S}^*$.
Decompose each in
$\mathcal T_{\mathcal S}^*=\mathcal V_{\mathcal S}^*
\oplus\mathcal K_{\mathcal S}^*$, and denote the identifiable components by
$\bm e_t^V$ and $\bm d_t^V$. Since
$\mathcal K_{\mathcal S}^*\subseteq\mathcal N_{\mathcal S}$, its components
leave all utilities and probabilities unchanged. Equation
\eqref{eq:fixed_mnl_null_space} therefore gives
\[
\bm G_t\bm d_t=\bm G_t\bm d_t^V,
\qquad
\nu_t^2=(\bm d_t^V)^\top\bm G_t\bm d_t^V,
\]
and $\bm q^t$ depends only on~$\bm d_t^V$. Thus any nonuniqueness of the QP
direction within~$\mathcal K_{\mathcal S}^*$ is immaterial to the bound.
Moreover, $\bm G^*$ is positive
definite on~$\mathcal V_{\mathcal S}^*$; otherwise a nonzero vector would
belong to both~$\mathcal V_{\mathcal S}^*$ and
$\mathcal K_{\mathcal S}^*$. Continuity then gives $m_{\mathcal S}>0$ on a
sufficiently small neighborhood.

Let $\bm Z$ be as in the proposition. Local face stability and QP stationarity
give
\[
\bm Z^\top\{\bm g_t+\bm G_t\bm d_t\}=\bm0,
\]
because $\bm Z^\top\bm A_{\Theta,\mathcal I^*}^\top=\bm0$ and the QP
multipliers vanish outside~$\mathcal I^*$. Using
\eqref{eq:fixed_mnl_null_space}, this is precisely the ordinary Newton equation
on~$\mathcal V_{\mathcal S}^*$. First-order optimality along the active face
also gives $\bm Z^\top\nabla\mathcal L_{\mathcal S}
(\bm\theta_{\mathcal S}^*)=\bm0$. The integral form of the projected score
expansion therefore gives
\[
\|\bm e_t^V+\bm d_t^V\|_2
\le
\frac{L_{G,\mathcal S}}{2m_{\mathcal S}}\|\bm e_t^V\|_2^2.
\]
After shrinking~$\mathcal U$ if needed, this implies
$\|\bm e_t^V\|_2\le2\|\bm d_t^V\|_2$. Since
$\nu_t^2\ge m_{\mathcal S}\|\bm d_t^V\|_2^2$, it follows that
\begin{equation}\label{eq:newton_endpoint_constant}
\|\bm e_t^V+\bm d_t^V\|_2
\le
\frac{2L_{G,\mathcal S}}{m_{\mathcal S}^2}\nu_t^2.
\end{equation}
The preceding Newton estimate gives
$\bm d_t^V=-\bm e_t^V+O(\|\bm e_t^V\|_2^2)$. Continuity of the probability
map and its derivative then gives $\bm q_i^t\to\bm\pi_i^*$, so the
neighborhood can be shrunk until $\bm q^t$ is feasible throughout.

Since
$D\bm\pi_i(\bm\theta_t)[\bm v]=\bm S_i^t\bm A_i\bm v$, where
$\bm S_i^t:=\operatorname{Diag}(\bm\pi_i^t)-\bm\pi_i^t(\bm\pi_i^t)^\top$
and $\bm A_i$ collects the row vectors $\bm a_{ij}^\top$, the probability-map
remainder satisfies
\[
\|\bm q_i^t-\bm\pi_i(\bm\theta_t+\bm d_t^V)\|_2
\le
\frac{L_{\pi,\mathcal S}}{2}\|\bm d_t^V\|_2^2
\le
\frac{L_{\pi,\mathcal S}}{2m_{\mathcal S}}\nu_t^2.
\]
Combining this inequality with~\eqref{eq:newton_endpoint_constant} and the
definition of~$M_{\pi,\mathcal S}$ proves
\[
\|\bm q_i^t-\bm\pi_i^*\|_2
\le
\left(
\frac{L_{\pi,\mathcal S}}{2m_{\mathcal S}}
+
\frac{2M_{\pi,\mathcal S}L_{G,\mathcal S}}{m_{\mathcal S}^2}
\right)\nu_t^2.
\]

Finally, applying~\eqref{eq:constrained_gap_identity} at
$\bm\theta_{\mathcal S}^*$ and using Proposition~\ref{prop:newton_linear_bound}
yields the exact constrained gap
\[
H(\mathcal D_s)-\mathrm{LB}_t(\mathcal D_s)
=
\sum_{i\in\mathcal S}
\mathrm{KL}\!\left(\bm q_i^t\,\middle\|\,\bm\pi_i^*\right)
+
\bm\mu_t^\top
(\bm b_\Theta-\bm A_\Theta\bm\theta_{\mathcal S}^*).
\]
Local face stability makes the multiplier term zero. KL divergence is locally
quadratic around the positive vector $\bm\pi_i^*$. More precisely,
$\mathrm{KL}(\bm q_i^t\|\bm\pi_i^*)
\le\|\bm q_i^t-\bm\pi_i^*\|_2^2/\pi_{\min,\mathcal S}^*$.
Summing the resulting inequalities proves the stated fourth-order bound, and
convergence follows immediately.
\end{proof}

\subsection{Proof of Proposition~\ref{prop:cross_set_propagation}}\label{app:cross_set_propagation}

\begin{proof}[Proof of Proposition~\ref{prop:cross_set_propagation}]
Consider part~(ii). The construction assigns exactly one probability vector to
every observation in~$\mathcal S$. On observations not covered by
$\mathcal G$, $\widetilde{\bm q}_i=\bm y_i$ has zero entropy and zero score
residual. Since the sets in~$\mathcal G$ are pairwise disjoint,
\[
\begin{aligned}
r_{\mathcal S}(\widetilde{\bm q})
&=
\sum_{\mathcal E\in\mathcal G}
r_{\mathcal E}(\bm q^{\mathcal E})
=
-\bm A_\Theta^\top\widetilde{\bm\mu},\\
\mathcal H_{\mathcal S}(\widetilde{\bm q})
&=
\sum_{\mathcal E\in\mathcal G}
\mathcal H_{\mathcal E}(\bm q^{\mathcal E}).
\end{aligned}
\]
Moreover, $\widetilde{\bm\mu}\ge\bm0$, and for every
$\bm\theta\in\Theta$,
\[
r_{\mathcal S}(\widetilde{\bm q})^\top\bm\theta
=
-\widetilde{\bm\mu}^\top\bm A_\Theta\bm\theta
\ge
-\widetilde{\bm\mu}^\top\bm b_\Theta.
\]
Lemma~\ref{lem:fenchel_leaf_bound} therefore gives
\[
H(\mathcal D_s)
\ge
\lambda N_0
+\mathcal H_{\mathcal S}(\widetilde{\bm q})
-\bm b_\Theta^\top\widetilde{\bm\mu},
\]
which is the stated sum. Part~(i) follows by taking
$\mathcal G=\{\mathcal E\}$.
\end{proof}

\subsection{Correctness of the Bound-Pruned Recursion}\label{app:pruned_dp_correctness}

The capped Bellman comparisons in Section~\ref{sec:full_pruned_dp} are valid because
$L_a\le Q_a$ and, for each split,
$L_0+L_1\le V(\mathcal D_s^{f=0},d+1)+V(\mathcal D_s^{f=1},d+1)=Q_f$.
Backward induction on the remaining depth shows that the leaf solver's
cap guarantee is preserved by the Bellman recursion. At a leaf,
Lemma~\ref{lem:fenchel_leaf_bound} and
Propositions~\ref{prop:newton_linear_bound} and
\ref{prop:cross_set_propagation} make every accepted direct or propagated
bound valid. An infeasible auxiliary probability vector causes no return; the
solver simply advances. Hence a leaf call returns its exact cost when it is
below~$U$ and returns \textsf{PRUNED} only when a valid bound, or the exact
cost, is at least~$U$.

Now consider a split under cap~$U$, with valid child-state bounds
$L_0$ and~$L_1$. If $L_0+L_1\ge U$, the split is pruned immediately. Otherwise
the first child is called under~$U-L_1$. A pruning return proves that its value
is at least~$U-L_1$, and hence that the split value is at least~$U$. If the
first child returns $v_0<U-L_1$, the second is called under~$U-v_0$. A pruning
return again proves that the split value is at least~$U$; two exact returns give
the exact split value below~$U$.

At a state, the algorithm retains exact action values below the incoming cap
and tightens that cap whenever it finds a better action. An action pruned under
the tightened cap cannot improve the incumbent. After
each action, the refreshed state bound~$L_V$ is compared with the current
cap. If $L_V$ reaches an exact incumbent, that incumbent is the Bellman
minimum and the remaining actions are skipped. If no exact incumbent exists
and $L_V$ reaches the incoming cap, the state returns \textsf{PRUNED}.
Otherwise evaluation continues. The root cap is $+\infty$, so the root
returns $V(\mathcal D,0)$.

Finally, before a lower-bound return, Algorithm~\ref{alg:progressive_leaf_fit}
stores the current iterate, MNL context, pending SQP step, and running
bound. A resumed call therefore advances to the first checkpoint not yet
evaluated. After any sequence of calls, the evaluated checkpoints are exactly
a prefix of the uninterrupted constrained Newton solve. Exact completion
evaluates the full path; an interim lower-bound return stops at a strict
prefix.

\subsection{Proof of Proposition~\ref{prop:pricing_structure}}\label{app:pricing}

\begin{proof}[Proof of Proposition~\ref{prop:pricing_structure}]

Fix a leaf~$k$ and suppress the leaf index. Write
$A_j=e^{\alpha_{jk}+s}$, $\gamma_j=\gamma_{jk}$, and denote product~$j$'s
feasible price interval by $[\underline p_j,\bar p_j]$. For prices~$\bm p$, let
\[
\pi_j(\bm p,s)
=
\frac{A_j e^{-\gamma_j p_j}}
{1+\sum_{\ell\in\mathcal J}A_\ell e^{-\gamma_\ell p_\ell}},
\qquad
\pi_0(\bm p,s)=1-\sum_{j\in\mathcal J}\pi_j(\bm p,s),
\]
and $R(\bm p,s)=\sum_jp_j\pi_j(\bm p,s)$.

\medskip\noindent\textbf{Scalar sufficiency.}\enskip
For a fixed leaf, the revenue objective depends on~$\bm{x}$ only through
$A_j=e^{\alpha_j+\bm{\beta}^{\top}\bm{x}}=e^{\alpha_j+s}$, and hence only through~$s$.
Two feature vectors routed to the same leaf with the same scalar index therefore induce the same pricing problem and the same optimizer.

\medskip\noindent\textbf{Clipped markup path.}\enskip
Existence follows from continuity of $R(\bm p,s)$ and compactness of the price
box. Let $\bm p^*$ be a global maximizer and write
$R^*=R(\bm p^*,s)$. The MNL revenue gradient is
\[
\frac{\partial R}{\partial p_j}(\bm p,s)
=
\pi_j(\bm p,s)\{1-\gamma_j(p_j-R(\bm p,s))\}.
\]
Introducing multipliers $\bar\lambda_j\ge0$ and $\underline\lambda_j\ge0$ for
the upper and lower bounds, stationarity gives
\[
\pi_j^*\{1-\gamma_j(p_j^*-R^*)\}-\bar\lambda_j+\underline\lambda_j=0,
\qquad j\in\mathcal J .
\]
Since $\pi_j^*>0$, complementary slackness implies the clipped
representation
\[
p_j^*
=
\operatorname{clip}\!\left(1/\gamma_j+R^*,\underline p_j,\bar p_j\right),
\qquad j\in\mathcal J,
\]
where $\operatorname{clip}(x,a,b) := \min\{\max\{x,a\},b\}$.
Thus every global maximizer lies on a clipped common-markup path with markup
$\mu=R^*$.

\medskip\noindent\textbf{Unique fixed point.}\enskip
For any scalar~$\mu$, define
\[
p_j(\mu)=\operatorname{clip}(1/\gamma_j+\mu,\underline p_j,\bar p_j),
\qquad
R(\mu,s)=R(\bm p(\mu),s),
\qquad
F(\mu,s)=R(\mu,s)-\mu .
\]
The function $F(\cdot,s)$ is continuous, $F(0,s)\ge0$, and
$F(\bar p_{\max},s)\le0$ because
$R(\bm p,s)\le \bar p_{\max}\sum_j\pi_j(\bm p,s)\le\bar p_{\max}$.
Hence a root exists on $[0,\bar p_{\max}]$.

It remains to show uniqueness. On any interval of~$\mu$ where the active set is
fixed, let $\mathcal U$ be the set of unclamped products and define
$\bar\gamma_{\mathcal U}:=\sum_{j\in\mathcal U}\gamma_j\pi_j$. Along this
interval, $dp_j/d\mu=1$ for $j\in\mathcal U$ and $dp_j/d\mu=0$ otherwise. The
MNL share derivative gives
\[
\frac{\partial\pi_j}{\partial\mu}
=
\pi_j\{\bar\gamma_{\mathcal U}-\gamma_j\mathbb I(j\in\mathcal U)\}.
\]
Therefore
\begin{align}
\frac{\partial R}{\partial\mu}
&=
\sum_{j\in\mathcal U}\pi_j
+\sum_{j\in\mathcal J}p_j
\pi_j\{\bar\gamma_{\mathcal U}-\gamma_j\mathbb I(j\in\mathcal U)\} \notag\\
&=
\bar\gamma_{\mathcal U}R-\mu\bar\gamma_{\mathcal U}
=
\bar\gamma_{\mathcal U}\{R(\mu,s)-\mu\}, \label{eq:key_identity}
\end{align}
where the second equality uses $\gamma_jp_j(\mu)=1+\gamma_j\mu$ for
$j\in\mathcal U$. Thus, within the same interval,
\[
\frac{\partial F}{\partial \mu}(\mu,s)
=
\bar\gamma_\mathcal{U}F(\mu,s)-1 .
\]
At any root of~$F$, the derivative is~$-1$. Immediately after such a root the
function is negative, and whenever $F<0$ its derivative is strictly negative.
By continuity at points where a price reaches or leaves a bound, $F$ cannot return to zero. The root
is therefore unique. Since every global maximizer must lie on the clipped path
and solve this fixed point, the pricing problem has a unique optimizer.

\medskip\noindent\textbf{Monotonicity and Lipschitz continuity.}\enskip
Let $\mu^*(s)$ denote the unique root. At the root,
\begin{equation}\label{eq:foc_identity}
R(\mu^*,s) = \mu^*.
\end{equation}
On any interval over which the same prices remain unclipped, the identity above gives
$F_\mu(\mu^*(s),s)=-1$. Also, because increasing~$s$ scales all product
attractiveness terms proportionally,
$\partial\pi_j/\partial s=\pi_j\pi_0$ at fixed~$\mu$, and hence
$F_s(\mu^*(s),s)=R_s(\mu^*(s),s)=\pi_0\mu^*(s)$. The implicit function theorem
therefore yields
\begin{equation}\label{eq:dmu_ds}
\frac{d\mu^*}{ds}=\pi_0\,\mu^*(s)\ge0 .
\end{equation}
The root is continuous in~$s$ by uniqueness and continuity of~$F$. Since the clipping path has only finitely many lower- and upper-bound thresholds, $\mu^*(s)$ is piecewise continuously differentiable, with matching one-sided derivatives at every threshold. Equation~\eqref{eq:dmu_ds} therefore extends monotonicity across these thresholds. Since the clipping map is nondecreasing, each
$p_j^*(s)=\operatorname{clip}(1/\gamma_j+\mu^*(s),\underline p_j,\bar p_j)$ is
nondecreasing in~$s$.

Finally, \eqref{eq:foc_identity} implies
$\mu^*(s)=R(\mu^*(s),s)\le \bar p_{\max}(1-\pi_0)$. Therefore, wherever the
derivative exists,
\[
\left|\frac{d\mu^*}{ds}\right|
\le
\pi_0(1-\pi_0)\bar p_{\max}
\le
\frac{\bar p_{\max}}{4}.
\]
Integrating this derivative bound over the finite partition induced by these thresholds gives
$|\mu^*(s)-\mu^*(s')|\le\bar p_{\max}|s-s'|/4$. The clipped price map is
$1$-Lipschitz in~$\mu$, which gives the stated $\ell_\infty$ bound for
$\bm p^*(s)$.
\end{proof}

\subsection{Proof of Theorem~\ref{thm:likelihood_pricing}}\label{app:likelihood_pricing}

\begin{proof}[Proof of Theorem~\ref{thm:likelihood_pricing}]
Fix leaf~$k$ and suppress its index where no confusion can arise.

\medskip\noindent\textbf{Restricted likelihood growth.}\enskip
Correct specification gives
\[
\overline{\mathcal L}_k(\bm\theta)
-\overline{\mathcal L}_k(\bm\theta_k^0)
=
\mathbb E\!\left[
\operatorname{KL}
\{\bm\pi_k(\bm X,\bm P;\bm\theta_k^0)
\,\|\,
\bm\pi_k(\bm X,\bm P;\bm\theta)\}
\,\middle|\,
\bm X\in\mathcal R_k
\right].
\]
The moment condition on~$\bm P$, compactness of~$\mathcal X$, and the smooth
MNL likelihood permit differentiation under the expectation on a bounded
neighborhood of~$\bm\theta_k^0$. Hence
$\nabla\overline{\mathcal L}_k(\bm\theta_k^0)=\bm0$, including when
$\bm\theta_k^0$ lies on the boundary of~$\Theta$.

Let $\bm d=\bm\theta-\bm\theta_k^0$. Convexity of~$\Theta$ keeps the
segment $\bm\theta_k^0+t\bm d$, $t\in[0,1]$, feasible, and
$\bm d\in T_\Theta(\bm\theta_k^0)$. By continuity of the Hessian and
\eqref{eq:restricted_fisher}, there is a neighborhood $U_k^{\mathcal L}$ of
$\bm\theta_k^0$ on which
\[
\bm d^\top
\nabla^2\overline{\mathcal L}_k(\bm\vartheta)\bm d
\ge
\frac{m_k}{2}\|\bm d\|_2^2
\]
for every~$\bm\vartheta$ on such a segment. Taylor's integral formula then
yields
\begin{equation}\label{eq:restricted_likelihood_growth}
\overline{\mathcal L}_k(\bm\theta)
-\overline{\mathcal L}_k(\bm\theta_k^0)
\ge
\frac{m_k}{4}\|\bm\theta-\bm\theta_k^0\|_2^2.
\end{equation}

\medskip\noindent\textbf{Uniform price stability.}\enskip
For $\bm\theta=(\bm\alpha,\bm\beta,\bm\gamma)$ and scalar~$\mu$, define
\[
p_j(\mu;\bm\theta)
=
\operatorname{clip}
\left(
\frac{1}{\gamma_j}+\mu,
\underline p_j,
\bar p_j
\right),
\qquad
F(\mu;\bm\theta,\bm x)
=
R_k\{\bm p(\mu;\bm\theta),\bm x;\bm\theta\}-\mu.
\]
The argument for Proposition~\ref{prop:pricing_structure} applies to every
$\bm\theta\in\Theta$: $F$ has a unique root
$\mu^*(\bm\theta,\bm x)\in[0,\bar p_{\max}]$, and the corresponding clipped
prices form~$\bm p_k^*(\bm x;\bm\theta)$. On an interval with fixed set
$\mathcal U$ of unclipped products,
\begin{equation}\label{eq:parameter_root_derivative}
\frac{\partial F}{\partial\mu}
=
\left(\sum_{j\in\mathcal U}\gamma_j\pi_{jk}\right)F-1.
\end{equation}

Joint continuity of~$F$, compactness of the common root interval, and root
uniqueness imply continuity of
$(\bm\theta,\bm x)\mapsto\mu^*(\bm\theta,\bm x)$. Choose a compact
neighborhood~$\mathcal K_k$ of~$\bm\theta_k^0$ relative to~$\Theta$ and let
$\bar\gamma=\max_{\bm\theta\in\mathcal K_k,j\in\mathcal J}\{\gamma_j\}$.
The root graph over~$\mathcal K_k\times\mathcal X$ is compact, so uniform
continuity gives a common neighborhood of this graph on which
$|F|\le1/(2\bar\gamma)$. After shrinking the parameter neighborhood if
necessary, \eqref{eq:parameter_root_derivative} therefore gives
$\partial F/\partial\mu\le-1/2$ between the roots associated with
$\bm\theta_k^0$ and~$\bm\theta$. This inequality extends across the finitely
many clipping breakpoints because~$F$ is absolutely continuous in~$\mu$.

On the same compact neighborhood, let $L_{F,k}$ be a uniform Lipschitz constant
for~$F$ in~$\bm\theta$. Comparing the two roots gives
\[
|\mu^*(\bm\theta,\bm x)-\mu^*(\bm\theta_k^0,\bm x)|
\le
2L_{F,k}\|\bm\theta-\bm\theta_k^0\|_2.
\]
Because $\gamma_j\ge\underline\gamma$, one may take
\[
L_{w,k}
=
\underline\gamma^{-2}
+2\sqrt{|\mathcal J|}\,L_{F,k}
\]
in~\eqref{eq:latent_markup_lipschitz}. Coordinatewise projection onto~$\mathcal P$
is nonexpansive, and hence
\begin{equation}\label{eq:price_parameter_stability}
\sup_{\bm x\in\mathcal X}
\|\bm p_k^*(\bm x;\bm\theta)
-\bm p_k^*(\bm x;\bm\theta_k^0)\|_2
\le
L_{w,k}\|\bm\theta-\bm\theta_k^0\|_2.
\end{equation}

\medskip\noindent\textbf{Quadratic pricing loss.}\enskip
Write $\bm p^0=\bm p_k^*(\bm x;\bm\theta_k^0)$,
$\widehat{\bm p}=\bm p_k^*(\bm x;\bm\theta)$, and
$\Delta\bm p=\widehat{\bm p}-\bm p^0$. Taylor's theorem gives
\begin{equation}\label{eq:pricing_loss_taylor}
\begin{aligned}
&R_k(\bm p^0,\bm x;\bm\theta_k^0)
-R_k(\widehat{\bm p},\bm x;\bm\theta_k^0)\\
&\qquad\le
\left|
\nabla_{\bm p}R_k(\bm p^0,\bm x;\bm\theta_k^0)^\top
\Delta\bm p
\right|
+\frac{M_{R,k}}{2}\|\Delta\bm p\|_2^2.
\end{aligned}
\end{equation}
For any closed interval~$I$ and all $u,v\in\mathbb R$, Euclidean projection
satisfies
\begin{equation}\label{eq:interval_projection}
|\{u-\Pi_I(u)\}\{\Pi_I(v)-\Pi_I(u)\}|
\le
|u-v|^2.
\end{equation}
Indeed, the left side vanishes when $u\in I$; otherwise it can be nonzero only
when~$v$ crosses the nearest endpoint, in which case both factors are bounded
by~$|u-v|$.

Let $\bm w^0=\bm w_k(\bm x;\bm\theta_k^0)$,
$\widehat{\bm w}=\bm w_k(\bm x;\bm\theta)$, and
$\pi_{jk}^0:=\pi_{jk}(\bm x,\bm p^0;\bm\theta_k^0)$. At the true optimum,
\[
\frac{\partial R_k}{\partial p_j}
(\bm p^0,\bm x;\bm\theta_k^0)
=
\pi_{jk}^0\gamma_{jk}^0(w_j^0-p_j^0).
\]
Applying~\eqref{eq:interval_projection} with
$I=[\underline p_j,\bar p_j]$, $u=w_j^0$, and $v=\widehat w_j$ yields
\[
\left|
\nabla_{\bm p}R_k(\bm p^0,\bm x;\bm\theta_k^0)^\top
\Delta\bm p
\right|
\le
\bar\gamma_k\|\widehat{\bm w}-\bm w^0\|_2^2.
\]
Projection nonexpansiveness also gives
$\|\Delta\bm p\|_2\le\|\widehat{\bm w}-\bm w^0\|_2$. Substitution into
\eqref{eq:pricing_loss_taylor} and use of
\eqref{eq:latent_markup_lipschitz} give
\begin{equation}\label{eq:quadratic_pricing_loss}
\begin{aligned}
0
&\le
R_k(\bm p^0,\bm x;\bm\theta_k^0)
-R_k(\widehat{\bm p},\bm x;\bm\theta_k^0)\\
&\le
L_{w,k}^2
\left(\bar\gamma_k+\frac{M_{R,k}}{2}\right)
\|\bm\theta-\bm\theta_k^0\|_2^2.
\end{aligned}
\end{equation}
The lower bound follows from the optimality of~$\bm p^0$ and does not require
the two price vectors to have the same active set.

Combining~\eqref{eq:restricted_likelihood_growth} and
\eqref{eq:quadratic_pricing_loss} proves~\eqref{eq:likelihood_pricing_constant}
on $U_k=U_k^{\mathcal L}\cap U_k^w$. For the global extension, convexity and
\eqref{eq:restricted_fisher} make~$\bm\theta_k^0$ the unique minimizer of
$\overline{\mathcal L}_k$ over~$\Theta$: a second minimizer would make the
objective constant on the feasible segment joining the two points, contradicting
positive curvature at~$\bm\theta_k^0$ in that direction. If $\Theta\setminus U_k$ is nonempty,
compactness gives
\[
\eta_k
=
\min_{\bm\theta\in\Theta\setminus U_k}
\{\overline{\mathcal L}_k(\bm\theta)
-\overline{\mathcal L}_k(\bm\theta_k^0)\}
>0.
\]
Since revenue lies in $[0,\bar p_{\max}]$, the theorem holds on all of~$\Theta$
with $C_k^{\mathrm{glob}}=\max\{C_k,\bar p_{\max}/\eta_k\}$. If
$\Theta\subseteq U_k$, take $C_k^{\mathrm{glob}}=C_k$.
\end{proof}

\subsection{Proof of Proposition~\ref{prop:lookup}}\label{app:lookup}

\begin{proof}[Proof of Proposition~\ref{prop:lookup}]

Fix a leaf~$k$ and suppress the leaf index; write
$[\underline s,\bar s]$ for $[\underline s_k,\bar s_k]$ and
$\Delta$ for $\Delta_k$.

\textbf{Part~(i).}
For any $s \in [\underline{s},\bar{s}]$, the nearest grid point~$s_m$ satisfies $|s - s_m| \le \Delta/2$.
By the Lipschitz bound in Proposition~\ref{prop:pricing_structure},
\[
\|\bm{p}^*(s) - \tilde{\bm{p}}(s)\|_\infty
\;\le\; \frac{\bar{p}_{\max}}{4}\,\frac{\Delta}{2}
\;=\; \frac{\bar{p}_{\max}}{8}\,\Delta.
\]

\medskip\noindent\textbf{Part~(ii).}
The proof works in the price vector space, where the revenue function
$R(\bm{p},s) = \sum_j p_j\,\pi_j(\bm{p},s)$ is $C^\infty$ in~$\bm{p}$
for any~$s$, since the MNL shares are smooth functions of the prices.
This avoids regime-crossing issues that arise in the one-dimensional markup space.

Write $\bm{p}^* = \bm{p}^*(s)$ and $\tilde{\bm{p}} = \bm{p}^*(s_m)$, and set $\varepsilon := \bar{p}_{\max}\Delta/8$.
Define $\delta\bm{p} := \tilde{\bm{p}} - \bm{p}^*$ and the path
$F(t) := R(\bm{p}^* + t\,\delta\bm{p},\; s)$ for $t\in[0,1]$.
Since both $\bm{p}^*$ and $\tilde{\bm{p}}$ lie in~$\mathcal P$,
so does the convex combination $\bm{p}(t) := (1-t)\bm{p}^* + t\,\tilde{\bm{p}}$ for all $t\in[0,1]$.
By part~(i), $\|\delta\bm{p}\|_\infty \le \varepsilon$.
By the integral form of Taylor's theorem:
\begin{equation}\label{eq:taylor_price}
F(1) = F(0) + F'(0) + \int_0^1(1-t)\,F''(t)\,dt.
\end{equation}
Since $\bm{p}^*$ maximizes $R(\cdot,s)$ over~$\mathcal P$,
the revenue loss $F(0) - F(1) \ge 0$.
We bound the two terms separately.

\medskip\noindent\textbf{Bounding $|F'(0)|$ (first-order term).}\enskip
Since $\frac{\partial R}{\partial p_j} = \pi_j\bigl[1 - \gamma_j(p_j - R)\bigr]$,
we have $F'(0) = \sum_j \pi_j\,\varphi_j\,\delta p_j$,
where $\varphi_j := 1 - \gamma_j(p_j^* - R^*)$ and $R^* = R(\bm{p}^*,s) = \mu^*(s)$
by~\eqref{eq:foc_identity}.
We classify each product~$j$:
\begin{itemize}
\item \textbf{Unclamped at~$s$}
($\underline{p}_j < p_j^* < \bar{p}_j$):
$p_j^* = 1/\gamma_j + \mu^*$,
so $\varphi_j = 1 - \gamma_j(1/\gamma_j + \mu^* - \mu^*) = 0$.
Contribution: $0$.
\item \textbf{Clamped at the same bound at both $s$ and~$s_m$}
(e.g., $p_j^* = \tilde{p}_j = \bar{p}_j$, or both equal~$\underline{p}_j$):
$\delta p_j = 0$.
Contribution: $0$.
\item \textbf{Regime-crossing} (clamped at~$s$ but with a different status at~$s_m$, either unclamped or clamped at a different bound):
Neither $\varphi_j$ nor $\delta p_j$ is zero, but both are $O(\Delta)$.
(If~$j$ is unclamped at~$s$ but clamped at~$s_m$, the first-order condition gives $\varphi_j = 0$,
so this case reduces to the first case above.)
\end{itemize}
For a regime-crossing product~$j$
(say, clamped at $p_j^* = \bar{p}_j$ at~$s$ and unclamped at~$s_m$),
there exists a regime boundary~$s_c$ between~$s$ and~$s_m$
where $\bar{p}_j = 1/\gamma_j + \mu^*(s_c)$,
giving $\varphi_j(s_c) = 1 - \gamma_j(1/\gamma_j) = 0$.
Therefore
\[
|\varphi_j|
= |1 - \gamma_j(\bar{p}_j - \mu^*(s))|
= \gamma_j\,|\mu^*(s) - \mu^*(s_c)|
\le \gamma_j\,\frac{\bar{p}_{\max}}{4}\,\frac{\Delta}{2}
= \gamma_j\,\varepsilon,
\]
where we used the Lipschitz bound on~$\mu^*$ established in the proof of Proposition~\ref{prop:pricing_structure}
and $|s - s_c| \le |s - s_m| \le \Delta/2$.
The same argument applies when $j$~is clamped at~$\underline{p}_j$
(with the boundary condition $\underline{p}_j = 1/\gamma_j + \mu^*(s_c)$, so $\varphi_j(s_c) = 0$ identically).
Combined with $|\delta p_j| \le \varepsilon$:
\begin{equation}\label{eq:F_prime_bound}
|F'(0)|
\;\le\; \sum_{j:\text{crossing}} \pi_j\,\gamma_j\,\varepsilon^2
\;\le\; \gamma_{\max}\,\varepsilon^2,
\end{equation}
where the last inequality uses $\sum_j \gamma_j\pi_j \le \gamma_{\max}\sum_j\pi_j \le \gamma_{\max}$.

\medskip\noindent\textbf{Bounding $|F''(t)|$ (second-order term).}\enskip
We have $F''(t) = \delta\bm{p}^\top\,\nabla^2 R(\bm{p}(t))\,\delta\bm{p}$.
At any price vector~$\bm{p}$ in the box, write
$\varphi_j(\bm{p}) := 1 - \gamma_j(p_j - R(\bm{p},s))$ for the marginal revenue factor.
The Hessian entries of~$R$ with respect to~$\bm{p}$ are, for $j\ne l$,
\[
H_{jl}
:= \frac{\partial^2 R}{\partial p_j\,\partial p_l}
= \pi_j\,\pi_l\bigl(\gamma_l\varphi_j + \gamma_j\varphi_l\bigr),
\]
and, for the diagonal,
\[
H_{jj}
= -\gamma_j\,\pi_j\bigl[\varphi_j(1-2\pi_j) + 1\bigr].
\]
Since $0 \le p_j \le \bar{p}_{\max}$ and $0 \le R \le \bar{p}_{\max}$
(the latter because $R = \sum_l p_l\pi_l \le \bar{p}_{\max}\sum_l\pi_l \le \bar{p}_{\max}$),
we have $|p_j - R| \le \bar{p}_{\max}$ and thus
$|\varphi_j| \le 1 + \gamma_{\max}\bar{p}_{\max}$
for all~$\bm{p}\in\mathcal P$.
The entry sums are bounded by
\begin{align*}
\textstyle\sum_j |H_{jj}|
&\;\le\; \gamma_{\max}(2 + \gamma_{\max}\bar{p}_{\max}), \\
\textstyle\sum_{j\ne l} |H_{jl}|
&\;\le\; 2\gamma_{\max}(1 + \gamma_{\max}\bar{p}_{\max}),
\end{align*}
where we used $\sum_j\gamma_j\pi_j \le \gamma_{\max}$
and $\sum_{j\ne l}\pi_j\pi_l(\gamma_j+\gamma_l) \le 2\gamma_{\max}$.
Therefore
\begin{equation}\label{eq:hessian_bound}
|F''(t)|
\;\le\; \|\delta\bm{p}\|_\infty^2\sum_{j,l}|H_{jl}|
\;\le\; \Gamma\,\varepsilon^2,
\quad\text{where}\quad
\Gamma := \gamma_{\max}(4 + 3\gamma_{\max}\bar{p}_{\max}).
\end{equation}

\medskip\noindent\textbf{Combining.}\enskip
From~\eqref{eq:taylor_price}, the revenue loss satisfies
\[
F(0) - F(1)
= -F'(0) - \int_0^1\!(1-t)\,F''(t)\,dt
\;\le\; |F'(0)| + \tfrac{1}{2}\max_{t\in[0,1]}\{|F''(t)|\}.
\]
Substituting \eqref{eq:F_prime_bound}, \eqref{eq:hessian_bound},
and $\varepsilon = \bar{p}_{\max}\Delta/8$:
\begin{align*}
R(\bm{p}^*,s) - R(\tilde{\bm{p}},s)
&\;\le\; \gamma_{\max}\,\varepsilon^2 + \tfrac{1}{2}\,\Gamma\,\varepsilon^2
\;=\; \bigl(\gamma_{\max} + \tfrac{1}{2}\Gamma\bigr)\,\varepsilon^2 \\
&\;=\; \tfrac{1}{2}\bigl(2\gamma_{\max} + \Gamma\bigr)\,\varepsilon^2
\;=\; \tfrac{1}{2}\gamma_{\max}\bigl(6 + 3\gamma_{\max}\bar{p}_{\max}\bigr)\,\varepsilon^2 \\
&\;=\; \frac{3\gamma_{\max}(2 + \gamma_{\max}\bar{p}_{\max})\,\bar{p}_{\max}^2}{128}\,\Delta^2.\qedhere
\end{align*}
\end{proof}

\section{Synthetic Data-Generating Process}\label{app:synth_dgp}

This appendix gives the full specification of the three synthetic regimes summarized in Section~\ref{sec:synth_dgp}. They share the same features, prices, and true tree and differ only in their leaf parameters.

\paragraph{Features and prices.}
Each instance consists of $N = 5{,}000$ independent observations. For each observation~$i$, we generate $|\mathcal{F}_{\textup{cont}}| = 5$ continuous features $x_1,\ldots,x_5 \overset{\mathrm{iid}}{\sim} \operatorname{Uniform}(0,1)$ and $|\mathcal{F}_{\textup{bin}}| = 4$ binary features $d_1,\ldots,d_4 \overset{\mathrm{iid}}{\sim} \operatorname{Bernoulli}(0.5)$, all mutually independent. The total feature count is $|\mathcal{F}| = 9$, of which five ($x_4, x_5, d_2, d_3, d_4$) are noise features that play no role in the true segmentation.

Each observation involves $|\mathcal{J}| = 2$ products with base prices $(p_1^{\mathrm{base}},p_2^{\mathrm{base}})=(20,12)$, reflecting moderate vertical differentiation. Observed prices are drawn as $p_{ij}\sim\operatorname{Uniform}(0.75p_j^{\mathrm{base}},1.25p_j^{\mathrm{base}})$ independently across observations, with a post-hoc swap to enforce $p_{i1}\geq p_{i2}$.

\paragraph{True tree structure.}
Preference heterogeneity is governed by a depth-3 asymmetric tree with five leaves, depicted in Figure~\ref{fig:dgp_tree}. The root splits on~$x_1$ at~$0.50$. The left branch ($x_1 < 0.50$) splits on the binary feature~$d_1$, producing segments~0 and~1 at depth~2. The right branch ($x_1 \geq 0.50$) splits sequentially on~$x_2$ and~$x_3$, producing segments~2, 3, and~4, with segment~4 reached only at depth~3. The tree is deliberately asymmetric: the left branch reaches depth~2 (one split below the root), while the right branch reaches depth~3 (two splits below the root). This asymmetry is designed to test whether each method correctly allocates depth budget to the more heterogeneous branch.

\begin{figure}[htb]
  \centering
  \includegraphics[width=0.78\linewidth]{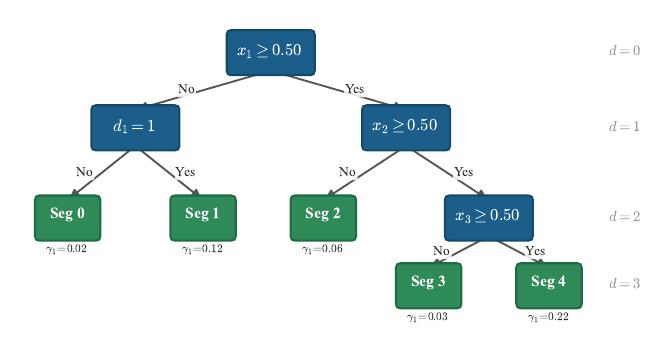}
  \caption{Common true tree in the three synthetic regimes: internal nodes (blue) show split conditions, and leaves (green) show segments with price sensitivity~$\gamma_1$ for product~1 at the Conflict endpoint. The asymmetric depth requires the tree construction method to allocate its depth budget across branches.\label{fig:dgp_tree}}
\end{figure}

\paragraph{MNL parameters.}
Conditional on segment assignment $k\in\{0,\ldots,4\}$, the systematic utility of product~$j$ is
\[
v_{ijk}=\alpha_{jk}+\bm\beta^\top\bm x_i-\gamma_{jk}p_{ij},
\]
with outside-option utility $v_{i0k}=0$. Choices are sampled from the resulting MNL probabilities, equivalently by adding independent standard-Gumbel shocks to all alternatives $j\in\{0\}\cup\mathcal J$. The feature coefficients $\bm{\beta} = (0.02, -0.01, 0.01, 0, \ldots, 0)^\top$ are common across segments and near zero, so that within-segment feature effects contribute minimally. Table~\ref{tab:dgp_params} reports the segment-specific parameters at the Conflict endpoint.

\begin{table}[H]
\TABLE
{Ground-truth MNL parameters at the Conflict endpoint.\label{tab:dgp_params}}
{\begin{tabular*}{\linewidth}{@{\extracolsep{\fill}}clrrrrr}
\toprule
& & \multicolumn{2}{c}{Intercept $\alpha_j$} & \multicolumn{2}{c}{Price sensitivity $\gamma_j$} & \\
\cmidrule(lr){3-4}\cmidrule(lr){5-6}
Seg.\ & Path & $\alpha_1$ & $\alpha_2$ & $\gamma_1$ & $\gamma_2$ & Purchase rate \\
\midrule
0 & $x_1\!<\!0.5,\; d_1\!=\!0$ & $0.20$ & $-1.00$ & $0.02$ & $0.12$ & 48\% \\
1 & $x_1\!<\!0.5,\; d_1\!=\!1$ & $-1.00$ & $\phantom{-}0.20$ & $0.12$ & $0.02$ & 50\% \\
2 & $x_1\!\geq\!0.5,\; x_2\!<\!0.5$ & $0.50$ & $-0.30$ & $0.06$ & $0.10$ & 42\% \\
3 & $x_1\!\geq\!0.5,\; x_2\!\geq\!0.5,\; x_3\!<\!0.5$ & $0.30$ & $-0.50$ & $0.03$ & $0.22$ & 44\% \\
4 & $x_1\!\geq\!0.5,\; x_2\!\geq\!0.5,\; x_3\!\geq\!0.5$ & $-0.50$ & $\phantom{-}0.30$ & $0.22$ & $0.03$ & 49\% \\
\bottomrule
\end{tabular*}}
{Feature coefficients $\bm{\beta} = (0.02, -0.01, 0.01, 0, \ldots, 0)^\top$ are common across segments. Purchase rate is the expected probability of choosing either product.}
\end{table}

\paragraph{Parameter regimes.}
Let $\bm\theta_k^{\mathrm C}$ denote the Conflict parameters in Table~\ref{tab:dgp_params}. To construct the Aligned endpoint, we first shrink the contrast between segments~0 and~1 by 25\% around their midpoint, component by component, for $\bm\alpha$, $\bm\beta$, and~$\bm\gamma$. We then add $0.75(1,-1)^\top$ to the intercept vector in both left-branch segments and subtract the same vector from each right-branch segment. Denote the resulting parameters by $\bm\theta_k^{\mathrm A}$. This transformation makes~$x_1$ the best one-step root split while preserving all three downstream branches after their BIC penalties.

For replication~$r$ in the Transition regime, the leaf parameters are
\[
\bm\theta_k^{(r)}=(1-w_r)\bm\theta_k^{\mathrm A}+w_r\bm\theta_k^{\mathrm C},
\qquad k=0,\ldots,4.
\]
The ten values $0.055,0.065,\ldots,0.145$ are randomly permuted across the ten seeds before either tree method is fitted. This interval straddles the population root-ranking transition: the one-step criterion favors~$x_1$ in five replications and~$d_1$ in five. A population audit based on expected MNL probabilities verifies that every true downstream split remains profitable after its BIC penalty throughout the interval. The Aligned and Conflict regimes use $\bm\theta_k^{\mathrm A}$ and $\bm\theta_k^{\mathrm C}$, respectively.

\paragraph{Adjusted Rand index.}
For the true and estimated partitions, let $n_{ab}$ be the number of test
observations in true segment~$a$ and estimated leaf~$b$, with margins
$n_{a\cdot}=\sum_b n_{ab}$ and $n_{\cdot b}=\sum_a n_{ab}$, and let
$n=\sum_{a,b}n_{ab}$. Writing
$E=\{\sum_a\binom{n_{a\cdot}}2\}
\{\sum_b\binom{n_{\cdot b}}2\}/\binom n2$, the adjusted Rand index is
\[
\operatorname{ARI}
=
\frac{\displaystyle\sum_{a,b}\binom{n_{ab}}2-E}
{\displaystyle
\frac12\left\{
\sum_a\binom{n_{a\cdot}}2+
\sum_b\binom{n_{\cdot b}}2
\right\}-E}.
\]

\section{Out-of-Sample Comparison in Airline Markets}\label{app:airline_oos}

The synthetic comparison in Section~\ref{sec:synth_results} uses a known MNL tree as the ground truth, allowing structural recovery to be measured directly. We test whether the performance ranking extends beyond that design using the pre-deployment records that trained the market-specific models in Section~\ref{sec:field}. The carrier collected these data from November~2024 through November~2025 across 48 airline OD markets, before the randomized experiment began.

Each record contains the prices offered for nine seat-group products, the resulting purchases, and features available at the time of pricing. The feature set includes days to departure, offer stage, ticket fare tier, and other booking and trip characteristics. Models are estimated separately by market, and market identities are omitted.

For this offline comparison, we reserve a stratified 20\% sample within each market for testing and estimate all methods on the remaining 80\%. The deployed models in Section~\ref{sec:field} use the full historical sample. All methods use the same features and parameter constraints: $-10\leq\alpha_{jk}\leq10$, $-5\leq\beta_{rk}\leq5$, and $10^{-4}\leq\gamma_{1k}\leq\cdots\leq\gamma_{9k}$ in each leaf~$k$, where $r$ indexes feature coefficients. For OCMT-MNL, we set $D=3$, $\lambda=\lambda_{\textup{AIC}}$, and $B=8$. GCMT-MNL uses the same tree settings as in Section~\ref{sec:synth_results}: maximum growth depth~14, minimum node weight~50, a quantile step of~0.05, and validation-based pruning. Its terminal models are then refitted on the full training sample.

Table~\ref{tab:airline_oos} reports equal-OD and observation-weighted OOS NLL. Figure~\ref{fig:airline_oos} shows the paired market-level differences and the confidence interval for their mean.

\begin{table}[!htb]
\TABLE
{Aggregate out-of-sample fit on historical airline data.\label{tab:airline_oos}}
{\begin{minipage}{\linewidth}
\begin{tabular*}{\linewidth}{@{\extracolsep{\fill}}lcc}
\toprule
Method & Equal-OD mean & Observation-weighted mean \\
\midrule
Single MNL & 0.6496 & 0.6379 \\
GCMT-MNL & 0.6465 & 0.6324 \\
OCMT-MNL & \textbf{0.6439} & \textbf{0.6312} \\
\bottomrule
\end{tabular*}
\end{minipage}}
{Lower values are better. OCMT-MNL sets $D=3$, $\lambda=\lambda_{\textup{AIC}}$, and $B=8$ in all 48 ODs. The observation-weighted mean covers 256,501 test observations.}
\end{table}

\begingroup
\setlength{\intextsep}{6pt}
\begin{figure}[H]
\centering
\captionsetup{skip=6pt}
\includegraphics[width=0.88\linewidth]{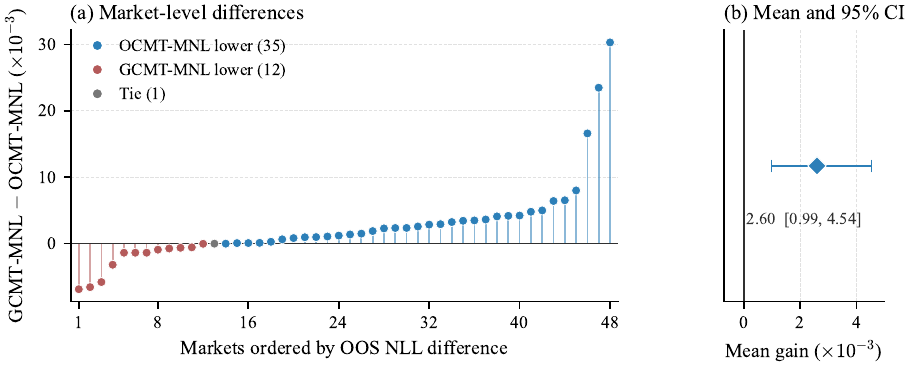}
\caption{Market-level OOS NLL comparison. \textbf{(a)}~Paired differences across the 48 markets, ordered from lowest to highest. Positive values favor OCMT-MNL. \textbf{(b)}~Equal-OD mean difference and 95\% confidence interval obtained by resampling markets.\label{fig:airline_oos}}
\end{figure}
\endgroup
\FloatBarrier

OCMT-MNL attains lower OOS NLL than GCMT-MNL in 35 of the 48 markets, with 12 losses and one numerical tie. For these counts, differences within $10^{-8}$ of zero per observation are treated as ties. Its equal-OD mean gain over GCMT-MNL is 0.00260 per observation, with a 95\% confidence interval of $[0.00099,0.00454]$ obtained by resampling markets. A paired two-sided Wilcoxon test gives $p=0.00052$. OCMT-MNL also has the lowest equal-OD and observation-weighted means among the three methods.

\section{Field-Experiment Sampling and Estimand}\label{app:field_estimand}

Repeated randomization and session-level sampling require care in defining the experimental contrast. The APO system randomizes the displayed policy independently at each login. After assignment, each session enters the data extract with probability 0.10, independently of its assigned policy and realized outcome. The extract is therefore a random sample of sessions rather than a panel containing every login of a sampled booking.

To state the implication precisely, index the logins of booking~$u$ by $t$, let $Z_{ut}$ denote the randomized policy, and let $S_{ut}$ indicate inclusion in the extract. Define $T_u=\min\{t:S_{ut}=1\}$. Conditional on randomization stratum~$g$ and $T_u<\infty$, independence of sampling and assignment gives
\[
\Pr\{Z_{uT_u}=a\mid g,T_u<\infty\}
=
\Pr\{Z_{ut}=a\mid g\}
\]
for every experimental policy~$a$. Thus selecting the earliest sampled login preserves the platform randomization: prior unobserved assignments and later reassignments are balanced across the first observed assignment. The comparison in Section~\ref{sec:field_design} consequently identifies the reduced-form effect associated with the policy displayed at that login, averaged over the remaining randomized exposure sequence.

This estimand is distinct from the effect of assigning one policy persistently throughout a booking and from the effect of the unobserved first platform exposure. Reassignment across visits mixes subsequent policy exposures and may therefore attenuate the contrast between persistent policy regimes. Independent reassignment preserves comparability across groups defined by first observed assignment. Our conditional randomization test respects the experimental design by permuting the first observed assignments within OD--calendar-week strata while preserving their observed counts; the booking bootstrap used for the confidence interval is stratified by OD and assignment. These choices follow standard guidance for repeated-unit online experiments \citep{AtheyImbens2017,KohaviTangXu2020}.

\section{Learned Segmentations in the Field Experiment}\label{app:segmentation}
The field experiment measures policy performance. The fitted models also reveal whether compact trees yield economically meaningful segmentation. We therefore summarize the split-sample tree structures and the full-data leaf policies across the 48 markets.

\paragraph{Tree size and sample size.}
The fitted trees have a median of 5 leaves (interquartile range $[3,6]$) and median depth 3. Figure~\ref{fig:48od_seg}(b) plots each OD's leaf count against $\log_{10}$ of its sample size; the Spearman correlation is $\rho=0.77$. Markets above the median training-sample size have a median of 6 leaves, compared with 4 below the median. Under the common depth cap and AIC penalty, larger samples can support finer segmentation without route-specific depth choices.

\paragraph{Selected split variables.}
Two of the eight candidate features account for roughly two thirds of all splits across the 48 fitted trees (Figure~\ref{fig:48od_seg}(a), bars $\mathrm{F1}$ and $\mathrm{F2}$). Days to departure (DTD, $\mathrm{F1}$) accounts for about 40\% of splits. This variable corresponds to the advance-purchase fences used in revenue-management practice, but OCMT-MNL selects its thresholds separately by market.

The offer-stage indicator ($\mathrm{F2}$) contributes 26\% of splits and is the root feature in 27 of the 48 trees. It distinguishes offers shown during ticket booking from those shown during a post-booking revisit. These split frequencies alone do not identify the behavioral mechanism behind the differences.

\begin{figure}[H]
\centering
\includegraphics[width=0.80\linewidth]{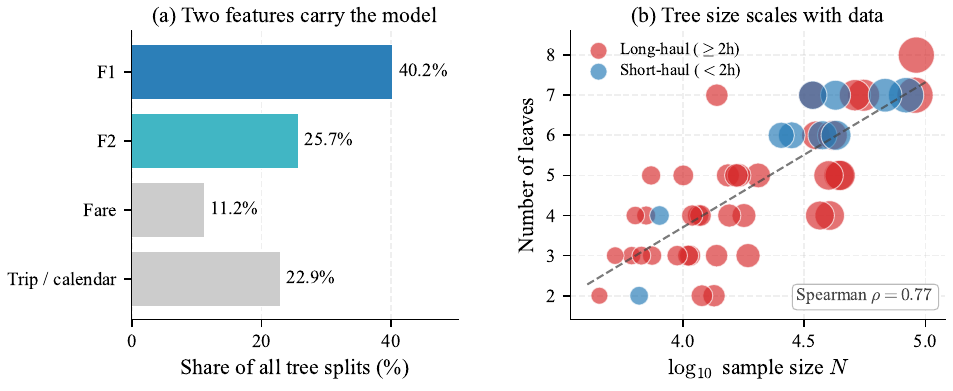}
\caption{Segmentations across the 48 OD markets. \textbf{(a)}~Split shares ($\mathrm{F1}$ = days to departure; $\mathrm{F2}$ = offer stage). \textbf{(b)}~Leaf count versus $\log_{10}N$; marker size represents~$N$, color denotes haul category, and $\rho=0.77$.\label{fig:48od_seg}}
\end{figure}

\paragraph{Leaf models and prescribed prices.}
Using the full-data deployment models, let $\widehat\gamma_{mjk}$ be the estimated price sensitivity for market~$m$, product~$j$, and leaf~$k$, and let $\bar r_{mjk}$ be the average ratio of recommended prices to static prices, weighting lookup grid points and both offer stages uniformly. Among the 360 market--product pairs with at least three leaves and nonzero variation in both quantities, 98.1\% have a negative cross-leaf Spearman correlation, 88.1\% have a correlation at most $-0.5$, and the median is $-0.80$. The median cross-leaf range of $\bar r_{mjk}$ is 14.2 percentage points, with 75.5\% exceeding two percentage points. The leaf models therefore induce material price differentiation across segments.

For Figure~\ref{fig:segment_price_profiles}, let $\mathcal I_k$ contain the historical observations routed to segment~$k$, and let $G$ denote the premium or standard group. We compute $\varepsilon_k^G$ at each observation's stage-specific static prices by applying a common proportional price increase to all products in~$G$, then aggregate predicted group demand over~$\mathcal I_k$. Panel~(b) evaluates the deployed lookup table for the same observations and reports one adjustment for each request--product pair. Thus both panels use the same routing, observations, and leaf parameters. In Segment~5, each premium product has a constant price across the historical observations. The premium-price spread therefore reflects differences between products, not price variation across requests for a given product. The elasticity is local to the static prices, whereas the prescribed prices solve the joint nine-product problem and need not vary monotonically with either group elasticity.

\end{document}